\documentclass[a4paper,10pt]{amsart}
\usepackage[english]{babel}
\usepackage{amsmath,amssymb,amsthm}
\usepackage{multicol}

\usepackage[pdftex]{color}
\usepackage[bookmarks=true,hyperindex,pdftex,colorlinks, citecolor=blue,linkcolor=blue, urlcolor=blue]{hyperref}

\theoremstyle{plain}
\newtheorem{theorem}{Theorem}[section]
\newtheorem{corollary}[theorem]{Corollary}
\newtheorem{prop}[theorem]{Proposition}
\newtheorem{proposition}[theorem]{Proposition}
\newtheorem{lemma}[theorem]{Lemma}

\theoremstyle{definition}

\newtheorem{remark}[theorem]{Remark}
\newtheorem{example}[theorem]{Example}

\newtheorem{definitions}[theorem]{Definitions}
\newtheorem{problem}[theorem]{Problem}

 \DeclareMathOperator{\Id}{\mathrm{Id}}

\DeclareMathOperator{\supp}{supp}

\newcommand{\C}{\mathbb{C}}
\newcommand{\R}{\mathbb{R}}
\newcommand{\N}{\mathbb{N}}

\newcommand{\eps}{\varepsilon}

\renewcommand{\le}{\leqslant}
\renewcommand{\geq}{\geqslant}
\renewcommand{\leq}{\leqslant}
\renewcommand{\ge}{\geqslant}

\newcommand{\tri}{\vert \hspace{-1.5pt} \vert  \hspace{-1.5pt} \vert }

\begin{document}
\title[On a second numerical index for Banach spaces]{On a second numerical index for Banach spaces}

\author[Kim]{Sun Kwang Kim}
\address[Kim]{Department of Mathematics, Kyonggi University, Suwon 443-760, Republic of Korea.\newline
\href{http://orcid.org/0000-0002-9402-2002}{ORCID: \texttt{0000-0002-9402-2002}  }}
\email{\texttt{sunkwang@kgu.ac.kr}}

\author[Lee]{Han Ju Lee}
\address[Lee]{Department of Mathematics Education,
Dongguk University - Seoul, 100-715 Seoul, Republic of Korea.
\href{http://orcid.org/0000-0001-9523-2987}{ORCID: \texttt{0000-0001-9523-2987}  }
}
\email{\texttt{hanjulee@dongguk.edu}}

\author[Mart\'{\i}n]{Miguel Mart\'{\i}n}
\address[Mart\'{\i}n]{Departamento de An\'{a}lisis Matem\'{a}tico, Facultad de
 Ciencias, Universidad de Granada, 18071 Granada, Spain.
\href{http://orcid.org/0000-0003-4502-798X}{ORCID: \texttt{0000-0003-4502-798X} }
 }
\email{mmartins@ugr.es}

\author[Mer\'{\i}]{Javier Mer\'{\i}}
\address[Mer\'{\i}]{Departamento de An\'{a}lisis Matem\'{a}tico, Facultad de
 Ciencias, Universidad de Granada, 18071 Granada, Spain.
\href{http://orcid.org/0000-0002-0625-5552}{ORCID: \texttt{0000-0002-0625-5552} }
 }
\email{jmeri@ugr.es}

\dedicatory{Dedicated to Rafael Pay\'{a} on the occasion of his 60th birthday}

\date{April 19th, 2016}

\thanks{First author partially supported by Basic Science Research Program through the National Research Foundation of Korea (NRF) funded by the Ministry of Education, Science and Technology (2014R1A1A2056084). Second author partially supported by Basic Science Research Program through the National Research Foundation of Korea (NRF) funded by the Ministry of Education, Science and Technology (NRF-2014R1A1A2053875). Third and fourth authors partially supported by Spanish MINECO and FEDER grant MTM2015-65020-P and by Junta de Andaluc\'{\i}a and FEDER grant FQM-185.}

\subjclass[2010]{Primary 46B04; Secondary 46B20, 47A12}
\keywords{Banach space; numerical range; numerical index; skew hermitian operator; Bishop-Phelps-Bollob\'{a}s property for numerical radius}

\begin{abstract}
We introduce a second numerical index for real Banach spaces with non-trivial Lie algebra, as the best constant of equivalence between the numerical radius and the quotient of the operator norm modulo the Lie algebra. We present a number of examples and results concerning absolute sums, duality, vector-valued function spaces\ldots which show that, in many cases, the behaviour of this second numerical index differs from the one of the classical numerical index. As main results, we prove that Hilbert spaces have second numerical index one and that they are the only spaces with this property among the class of Banach spaces with one-unconditional basis and non-trivial Lie algebra. Besides, an application to the Bishop-Phelps-Bollob\'{a}s property for numerical radius is given.
\end{abstract}

\maketitle

\thispagestyle{plain}

\begin{center}
\begin{minipage}[c]{0,7\textwidth}\small
\tableofcontents
\end{minipage}
\end{center}

\section{Introduction}
Given a \textbf{real} Banach space $X$, $S_X$ is its unit sphere, $X^*$ is the topological dual space of $X$ and $\mathcal{L}(X)$ is the Banach algebra of all bounded linear operators on $X$. The \emph{numerical range} of $T\in \mathcal{L}(X)$ is given by
$$
V(T):=\bigl\{x^*(Tx)\,:\, (x,x^*)\in \Pi(X)\bigr\},
$$
where
$$
\Pi(X):=\bigl\{(x,x^*)\in X\times X^*\,:\, x\in S_{X},\,x^*\in S_{X^*},\, x^*(x)=1\bigr\}.
$$
The \emph{numerical radius} of $T\in \mathcal{L}(X)$ is given by
$$
v(T):=\sup\bigl\{|\lambda|\,:\,\lambda\in V(T)\bigr\}.
$$
The concept of numerical range was introduced independently by G.~Lumer and F.~Bauer in the 1960's, and it is an extension of the field of values of a matrix introduced by O.~Toeplitz in 1918. We refer the reader to the classical monographs \cite{B-D1,B-D2} for a reference. More recent references are
\cite[\S 2.1 and \S 2.9]{Cabrera-Rodriguez} and  \cite{Mar-numrange-JMAA2016}.
It is clear that $v$ is a continuous seminorm which satisfies $v(T)\leq \|T\|$ for every $T\in \mathcal{L}(X)$. Sometimes, $v$ is actually an equivalent norm, and to measure this possibility, we use the \emph{numerical index} of $X$:
\begin{align*}
n(X)& :=\inf\bigl\{v(T)\,:\, T\in \mathcal{L}(X),\, \|T\|=1\bigr\}\\
& = \max\bigl\{k\geq 0\,:\, k\|T\|\leq v(T)\ \forall T\in \mathcal{L}(X)\bigr\}.
\end{align*}
This concept was introduced in the 1970's \cite{D-Mc-P-W} and we will give in section~\ref{sect:survey-classicalnumindex} an overview of those known results about numerical index which will be relevant in our discussion. Observe that $n(X)=1$ if and only if $v(T)=\|T\|$ for every $T\in \mathcal{L}(X)$ and $n(X)=0$ if $v$ is not equivalent to the usual norm in $\mathcal{L}(X)$. There are Banach spaces with numerical index $0$ in which the numerical radius is a norm (see Example~\ref{example-old-nx=0zx=0}), but the more general situation for a Banach space $X$ with $n(X)=0$ is that there exists a non-null $T\in \mathcal{L}(X)$ with $v(T)=0$. To handle this situation, which will fully occur along the paper, we recall some notation. An operator $T\in \mathcal{L}(X)$ is \emph{skew-hermitian} if $v(T)=0$. The \emph{Lie algebra} of $X$ is the set $\mathcal{Z}(X)$ of all skew-hermitian operators on $X$, that is,
$$
\mathcal{Z}(X):=\bigl\{S\in \mathcal{L}(X)\,:\, v(S)=0\bigr\},
$$
which is a closed subspace of $\mathcal{L}(X)$. Then, in $\mathcal{L}(X)/\mathcal{Z}(X)$ we may consider two norms:
\begin{align*}
\|T+\mathcal{Z}(X)\|&:= \inf\bigl\{\|T-S\|\,:\, S\in \mathcal{Z}(X)\bigr\} \\
v\bigl(T+\mathcal{Z}(X)\bigr)&:= \inf\bigl\{v(T-S)\,:\,S\in \mathcal{Z}(X)\bigr\} = v(T)
\end{align*}
for all $T+\mathcal{Z}(X)\in \mathcal{L}(X)/\mathcal{Z}(X)$ (the last equality follows from the straightforward fact that $V(T-S)=V(T)$ for every $T\in \mathcal{L}(X)$ and every $S\in \mathcal{Z}(X)$). Observe that $v$ is a norm on $\mathcal{L}(X)/\mathcal{Z}(X)$, not merely a seminorm. It is immediate that $v(T)\leq \|T+\mathcal{Z}(X)\|$ for every $T\in \mathcal{L}(X)$, and we may define the \emph{second numerical index} of the Banach space $X$ as follows
\begin{align*}
n'(X)& :=\inf\bigl\{v(T)\,:\, T\in \mathcal{L}(X)\,, \|T+\mathcal{Z}(X)\|=1\bigr\}\\
& = \inf \left\{ \frac{v(T)}{\|T+\mathcal{Z}(X)\|}\,:\, T\in \mathcal{L}(X),\, T\notin \mathcal{Z}(X)\right\}\\
& = \max\bigl\{k\geq 0\,:\, k\|T+\mathcal{Z}(X)\|\leq v(T)\ \forall T\in \mathcal{L}(X)\bigr\}.
\end{align*}
It is clear that $n'(X)\in [0,1]$. The value $n'(X)=1$ means that $v$ and the quotient norm coincide on $\mathcal{L}(X)/\mathcal{Z}(X)$, while $n'(X)=0$ if $v$ and the quotient norm are non equivalent norms on $\mathcal{L}(X)/\mathcal{Z}(X)$. Of course, if $\mathcal{Z}(X)=\{0\}$ (in particular, if $n(X)>0$), then $n'(X)=n(X)$. Therefore, this new index would only  be interesting for Banach spaces with non-trivial Lie algebra. On the other hand, as $v$ is a norm in $\mathcal{L}(X)/\mathcal{Z}(X)$, it does not make sense to define a ``third numerical index'' of Banach spaces in the same way, since its value would be the same as the second numerical index.

Our goal here is to present a number of results related to this second numerical index. We start in section~\ref{sect:survey-classicalnumindex} with a compilation of known results about the ``classical'' numerical index which we will use along the paper. The first examples are contained in section~\ref{sect:FirstExamples}: finite-dimensional Banach spaces have positive second numerical index and Hilbert spaces have second numerical index one. An application is given in section~\ref{sect-application}: Hilbert spaces have the Bishop-Phelps-Bollob\'{a}s property for numerical radius. This answers a question possed in \cite{KLM-BPBp-nu}. Next, we study in section~\ref{sect-sums-moreexamples} the behaviour of the second numerical index with respect to absolute sums of two spaces showing, for instance, that the second numerical index of a $c_0$- or $\ell_p$-sum ($1\leq p\leq \infty$, $p\neq 2$) of Banach spaces is smaller than or equal to the infimum of the second numerical index of the summands. Some examples in which the reverse inequality is not true are also given. With all of this, we are able to present some interesting examples. For instance, we show that the set of values of the second numerical index contains the interval $[0,1/2]$. Section~\ref{sect-duality} is devoted to study the {duality relation}: it is shown that, in many cases, the second numerical index of the dual of a Banach space is less than or equal to the one of the space, and that there are examples in which this inequality is strict. Vector valued function spaces are studied in section~\ref{sec:vector-valued}: it is shown that the second numerical index of $C(K,X)$, $L_1(\mu,X)$ and $L_\infty(\mu,X)$ is not greater than $n'(X)$ (where $K$ is a Hausdorff compact topological space and $\mu$ is a positive measure). On the other hand, if $K$ contains more than one point and $H$ is a two-dimensional Hilbert space, then $n'(C(K,H))$ is not one. Section~\ref{sec:absolute-sums} is devoted to the study of Banach spaces having an absolute structure and of general absolute sums. We show that, in many cases, the second numerical index of a general absolute sum of Banach spaces is less than or equal to the infimum of the second numerical indices of the summands. As an application, we show that the set of values of $n'(X)$ when $X$ runs over all $3$-dimensional spaces with numerical index zero is not an interval. Some other results are given which will be used as tools in the next section. The main result of the paper is presented in section~\ref{sect-numindex-one}: Hilbert spaces are the unique Banach spaces with normalized one-unconditional basis which have numerical index zero and second numerical index one. Finally, we include an open problem section (section~\ref{sect:open_problems}).

\vspace*{2ex}

We finish the introduction presenting some notation. Let $X$ be a Banach space. For $x\in X$ and $x^*\in X^*$, we will use the notation $\langle x^*,x\rangle = x^*(x)$ when needed. If $H$ is a Hilbert space, then $(z\mid w)$ denotes the inner product of $z,w\in H$.

\section{Some known results about numerical index}\label{sect:survey-classicalnumindex}

Let us review here some known results about the numerical index of a Banach space which will be relevant in the further discussion on the second numerical index. We are not trying to provide an exhaustive list of all known results, but only those that are frequently used along the paper. We refer the reader to the survey paper \cite{KaMaPa} for a compilation about known results on numerical index of Banach spaces.

The first result shows that, even though the exact computation of numerical indices is complicated, it is possible to prove the existence of Banach spaces with any possible value of the numerical index.

\begin{prop}[\mbox{\cite[Theorem~3.6]{D-Mc-P-W}}]\label{prop:old-values}
$\bigl\{n(X)\, : \, X \text{ Banach space of dimension two} \, \bigr\}=[0,1]$.
\end{prop}

For $1\leq p \leq \infty$ and $n\in \N$, we write $\ell_p^{n}$ to denote the space $\R^n$ endowed with the $\ell_p$-norm. Even though the exact value of $n(\ell_p^2)$ is unknown for $1<p<\infty$, $p\neq 2$, the result above can be improved as follows.

\begin{example}[\mbox{\cite[Theorem~3.6]{D-Mc-P-W} and \cite[Theorem~1]{MarMerLAMA2009}}] \label{example:old-values-lp2}
$\bigl\{n(\ell_p^2)\, : \, 1<p<\infty,\, p\neq 2\bigr\}=]0,1[$, $n(\ell_2^2)=0$, $n(\ell_1^2)=n(\ell_\infty^2)=1$, and $\lim\limits_{p\to 2} n(\ell_p^2)=0$.
\end{example}

It is known that the numerical index is continuous with respect to the Banach-Mazur distance between (isomorphic) Banach spaces \cite[Proposition~2]{F-M-P}, and this result shows that the set of values of the numerical index up to equivalent renorming is always an interval. Moreover, this interval is never trivial and it is maximal in many cases. We say that a Banach space $X$ \emph{admits a long biorthogonal system} if there exists $\{(x_\lambda,x_\lambda^*)\}_{\lambda\in \Lambda}\subset X\times X^*$ such that $x_\lambda^*(x_\mu)=\delta_{\lambda,\mu}$ for $\lambda,\mu\in \Lambda$ and the cardinality of $\Lambda$ coincides with the density character of $X$. Examples of spaces admitting a long biorthogonal system are reflexive spaces and separable spaces.

\begin{prop}[\mbox{\cite[Theorems 9 and 10]{F-M-P}}] \label{prop-old:values-up-to-renorming}
Let $X$ be a Banach space.
\begin{enumerate}
\item[(a)] $\bigl\{ n(Y)\,:\, \text{$Y$ is isomorphic to $X$}\bigr\}\supseteq [0,1/3[$.
\item[(b)] If $X$ admits a long biorthogonal system, then $\bigl\{ n(Y)\,:\, \text{$Y$ is isomorphic to $X$}\bigr\}\supseteq [0,1[$.
\end{enumerate}
\end{prop}

The result above shows that, among a wide class of Banach spaces, every value of the numerical index but $1$ is isomorphically innocuous. On the other hand, it is known that not every Banach space can be renormed to have numerical index $1$.

\begin{prop}[\mbox{\cite[Corollary~4.10]{AvKadMarMerShe-SCD}}] \label{prop-old:dualcontains_l_1}
Let $X$ be an infinite-dimensional Banach space with $n(X)=1$. Then, $X^*\supseteq \ell_1$.
\end{prop}

There is a relation between the numerical index of a Banach space and the one of its dual.

\begin{prop}[\mbox{\cite[Proposition~1.2]{D-Mc-P-W}}] \label{prop:old-inequalityduality}
Let $X$ be a Banach space. Then $v(T)=v(T^*)$ for every $T\in \mathcal{L}(X)$. Therefore, $n(X^*)\leq n(X)$.
\end{prop}

It was shown in \cite{BKMW} that the inequality above can be strict, solving a long standing open problem. Actually, the inequality $n(X^*)\leq n(X)$ is the unique possible restriction on the values of $n(X^*)$ and $n(X)$, as the following results shows. It does not appear explicitly anywhere, but it is easily deductible from the given references.

\begin{example}[\mbox{\cite[Examples 3.1 and 3.3]{BKMW} and \cite[Proposition~4.2]{MarJFA2008}}] \label{example:old-duality}
There are Banach spaces $X$ such that $n(X^*)<n(X)$. Moreover, for every $\alpha,\beta\in [0,1]$ with $\alpha\geq \beta$, there exists a Banach space $X_{\alpha,\beta}$ such that $n(X_{\alpha,\beta})=\alpha$ and $n(X_{\alpha,\beta}^*)=\beta$.
\end{example}

Our next results deal with the computation of the numerical index. First, a result about direct sums. Given an arbitrary family
$\{X_\lambda\, : \, \lambda\in\Lambda\}$ of Banach spaces and $1\leq p\leq \infty$, we
denote by $\left[\oplus_{\lambda\in\Lambda}
X_\lambda\right]_{c_0}$ (respectively $\left[\oplus_{\lambda\in\Lambda}
X_\lambda\right]_{\ell_p}$) the $c_0$-sum (respectively $\ell_p$-sum) of the family. In case $\Lambda$ has just two
elements, we use the simpler notation $X\oplus_\infty Y$ or
$X\oplus_p Y$.

\begin{prop}[\mbox{\cite[Proposition~1 and Remarks~2]{MarPay}}] \label{prop-suma-oldMarPay2000} Let $\{X_\lambda\,:\, \lambda \in \Lambda\}$ be a
family of Banach spaces. Then
\begin{enumerate}
\item[(a)] $n\Bigl(\left[\oplus_{\lambda\in\Lambda}
X_\lambda\right]_{\ell_p} \Bigr)\leq \inf_{\lambda} \,n(X_\lambda)$ for every $1< p < \infty$.
\item[(b)] $n\Bigl(\left[\oplus_{\lambda\in\Lambda} X_\lambda\right]_{c_0}
\Bigr)= n\Bigl(\left[\oplus_{\lambda\in\Lambda}
X_\lambda\right]_{\ell_1} \Bigr)=
n\Bigl(\left[\oplus_{\lambda\in\Lambda}
X_\lambda\right]_{\ell_\infty} \Bigr)= \inf_{\lambda} \,n(X_\lambda)$.
\item[(c)] If $n(X_\lambda)>0$ for every $\lambda\in \Lambda$, then
    $$
    \mathcal{Z}\Bigl(\left[\oplus_{\lambda\in\Lambda} X_\lambda\right]_{c_0}\Bigr) = \mathcal{Z} \Bigl(\left[\oplus_{\lambda\in\Lambda} X_\lambda\right]_{\ell_1}\Bigr) = \mathcal{Z} \Bigl(\left[\oplus_{\lambda\in\Lambda} X_\lambda\right]_{\ell_\infty}\Bigr)=\{0\}.
    $$
\end{enumerate}
\end{prop}

As a consequence of this result and Proposition~\ref{prop:old-values}, the following interesting example follows.

\begin{example}[\mbox{\cite[Examples~3.b]{MarPay}}]\label{example-old-nx=0zx=0}
{\slshape There exists a Banach space $X$ with $n(X)=0$ such that $\mathcal{Z}(X)=\{0\}$.}
\end{example}

Proposition~\ref{prop-suma-oldMarPay2000}.a was extended to a wide class of arbitrary absolute sums of Banach spaces in \cite{MarMerPopRan}. We will introduce the necessary notation about arbitrary absolute sums in section~\ref{sec:absolute-sums}, but we would like to mention here a couple of results concerning absolute sums of pairs of Banach spaces. We recall that an \emph{absolute norm} is a norm $|\cdot|_a$ on $\R^2$ such that $|(1,0)|_a=|(0,1)|_a=1$ and $|(s,t)|_a=|(|s|,|t|)|_a$ for every $s,t\in \R$. Given two Banach spaces $Y$ and $W$ and an absolute norm $|\cdot|_a$, the \emph{absolute sum} of $Y$ and $W$ with respect to $|\cdot|_a$, denoted by $Y\oplus_a W$, is the Banach space $Y\times W$ endowed with the norm
$$
\|(y,w)\|=\bigl|(\|y\|,\|w\|)\bigr|_a \qquad \bigl(y\in Y,\ w\in W\bigr).
$$
For background on absolute sums the reader is referred to \cite{MPRY} and references therein. Natural examples of absolute sums are the $\ell_p$-sum $\oplus_p$ for $1\leq p \leq \infty$, associated to the $\ell_p$-norm in $\R^2$.

The first result is an extension of Proposition~\ref{prop-suma-oldMarPay2000}.a to absolute sums of two Banach spaces.

\begin{proposition}[\mbox{\cite[Corollary~2.4]{MarMerPopRan}}] \label{prop:old-absolutesumsinequality}
Let $Y$, $W$ Banach spaces and let $\oplus_a$ be an absolute sum. Then
$$
n(Y\oplus_a W)\leq \min\{n(Y),\,n(W)\}.
$$
\end{proposition}

The second result allows to calculate the Lie algebra of many absolute sums of two Banach spaces and will be very useful in our discussion.

\begin{lemma}[\mbox{\cite[Lemma~5]{Paya82}}] \label{lemma:lemma5Paya82}
Let $Y$, $W$ be Banach spaces and $\oplus_a$ be an absolute sum different from the $\ell_2$-sum. Then every element in $\mathcal{Z}(Y\oplus_a W)$ is diagonal (i.e.\ it commutes with the projections). Moreover, given $T\in \mathcal{Z}(Y\oplus_a W)$, there exist $S\in \mathcal{Z}(Y)$ and $R\in \mathcal{Z}(W)$ such that
$$
T(y,w)=(Sy,Rw) \qquad \bigl(y\in Y,\ w\in W\bigr).
$$
\end{lemma}

Next we deal with vector valued function spaces. We recall some common notation. Let $X$ be a Banach space. Given a compact Hausdorff topological space, $C(K,X)$ (resp. $C_w(K,X)$) is the Banach space of all continuous (resp. weakly continuous) functions from $K$ into $X$ endowed with the supremum norm. If $L$ is a locally compact Hausdorff topological space, $C_0(L,X)$ is the Banach space of all continuous functions from $L$ to $X$ vanishing at infinity. For a completely regular Hausdorff topological space $\Omega$, we write $C_b(\Omega,X)$ to denote the Banach space of all bounded continuous functions from $\Omega$ into $X$ with the supremum norm. If  $(A,\Sigma,\mu)$ is a measure space, a strongly measurable function from $A$ to $X$ is a pointwise limit of a sequence of countably valued Borel measurable functions. The space $L_1(\mu, X)$ is the Banach space of equivalence classes of those strongly measurable functions $f$ from $\Omega$ to $X$ for which
\[
 \|f\|_1= \int_\Omega \|f(t)\| \, d\mu(t)<\infty.
\]
We write $L_\infty(\mu, X)$ for the Banach space of all equivalence classes of strongly measurable functions $f$ from $\Omega$ to $X$ for which
\[
\|f\|_\infty= \inf \bigl\{\lambda \ge 0\,:\,  \|f(t)\| \leq \lambda \ \  a.e.\bigr\}< \infty.
\]

\begin{prop}[\rm \mbox{\cite[Theorems 5 and 8]{MarPay}},  \mbox{\cite[Theorem 2.3]{MarVil}}, \mbox{\cite[Theorem 2 and Proposition 5]{LopezMartinMeri}}, and \mbox{\cite[Corollary 3.3]{ChoiGarciaMaestreMartin-QJ}}]\label{prop:old-ck-l1-linfty}
Let $K$ be a compact Hausdorff topological space, let $L$ be a locally compact Hausdorff topological space, let $\Omega$ be a completely regular Hausdorff topological space,
and let $\mu$ be a positive $\sigma$-finite measure. Then, for every Banach space $X$,
$$
n(C(K,X))=n(C_w(K,X))=n(C_0(L,X))=n(C_b(\Omega,X))=n(L_1(\mu,X))= n(L_\infty(\mu,X))=n(X).
$$
\end{prop}

We would like to finish this section with two useful results to calculate numerical radius. The first one was actually used in the proofs of the result above.

\begin{lemma}[\mbox{\cite[Theorem~9.3]{B-D1}}]\label{lemma:old-projection-dense}
Let $X$ be a Banach space. Write $\pi_X:\Pi(X)\longrightarrow X$ for the natural projection, and let $\Gamma$ be a subset of $\Pi(X)$ such that $\pi_X(\Gamma)$ is dense in $S_X$. Then
$$
v(T)=\sup \bigl\{|x^*(Tx)|\,:\,(x,x^*)\in \Gamma\bigr\}
$$
for every $T\in \mathcal{L}(X)$.
\end{lemma}

The second result in this line is very recent and deals with a new numerical range of operators between different Banach spaces introduced by M.~Ardalani \cite{Ardalani}. Recall that given a Banach space $X$, a subset $A$ of $S_{X^*}$ is said to be \emph{$1$-norming} (for $X$) if $\|x\|=\sup\{|x^*(x)|\,:\, x^*\in A\}$ for every $x\in X$ or, equivalently, if the weak-star closed convex hull of $A$ is the whole $B_{X^*}$.

\begin{lemma}[\mbox{\cite[Remark~2.6 and Proposition~3.1]{Mar-numrange-JMAA2016}}]\label{lemma:old-Anorming-Ardalani}
Let $X$ be a Banach space and let $A\subset S_{X^*}$ be $1$-norming. Then
$$
v(T)=\inf_{\delta>0}\sup\left\{|x^*(Tx)|\,:\ x\in S_X,\, x^*\in A,\, x^*(x)>1-\delta\right\}
$$
for every $T\in \mathcal{L}(X)$.
\end{lemma}

\section{First examples}\label{sect:FirstExamples}

\begin{proposition}\label{prop:finite-dimensional-n'>0}
If $X$ is a finite-dimensional space, then $n'(X)>0$.
\end{proposition}

\begin{proof}
This is an obvious consequence of the finite-dimensionality of $\mathcal{L}(X)/\mathcal{Z}(X)$, from which follows that the two norms $\|\cdot+\mathcal{Z}(X)\|$ and $v(\cdot)$ have to be equivalent.
\end{proof}

In the infinite-dimensional case, we may find examples of spaces with $n'(X)=0$.

\begin{example}\label{example-n'=0}
{\slshape There are infinite-dimensional Banach spaces $X$ with $n'(X)=0$.}
\end{example}

\begin{proof}
It is shown in Example~\ref{example-old-nx=0zx=0} that there exists a Banach space $X$ with $n(X)=0$ but such that $\mathcal{Z}(X)=\{0\}$. Then, we obviously have that $n'(X)=n(X)=0$.
\end{proof}

The main example here is the somehow surprising fact that Hilbert spaces have second numerical index one. We will see in section~\ref{sect-numindex-one} that, among Banach spaces with normalized one-unconditional basis, Hilbert spaces are the only ones having classical numerical index zero and second numerical index one.

\begin{theorem}\label{thrm-Hilbert}
Let $H$ be a Hilbert space. Then $n'(H)=1$.
\end{theorem}

We will use the following known results, which we prove here for the sake of completeness. Recall that in a real Hilbert space $H$ endowed with an inner product $(\cdot\mid\cdot)$, $H^*$ identifies with $H$ by the isometric isomorphism $x\longmapsto (\cdot\mid x)$. Therefore, $\Pi(H)=\{(x, x)\in H\times H\,;\, x\in S_H\}$ and so for every $T\in \mathcal{L}(H)$, one has $v(T)=\sup\{|(Tx\mid x)|\,:\,x\in S_H\}$.

\begin{lemma}\label{lemma-Hilbert}
Let $H$ be a Hilbert space.
\begin{enumerate}
\item[(a)] $\mathcal{Z}(H)=\bigl\{T\in \mathcal{L}(H)\,:\,T=-T^*\bigr\}$.
\item[(b)] If $T\in \mathcal{L}(H)$ is selfadjoint (i.e.\ $T=T^*$), then $\|T\|=v(T)$.
\end{enumerate}
\end{lemma}

\begin{proof}
(a). For every $T\in \mathcal{L}(H)$ we have
$$
2\bigl((Tx\mid y)+(x\mid Ty)\bigr) = (T(x+y)\mid x+y) - (x-y\mid T(x-y)).
$$
Now, if $v(T)=0$, then $(Tz\mid z)=0$ for every $z\in H$, so it follows that $(Tx\mid y)=-(x\mid Ty)$ for every $x,y\in H$, that is, $T^*=-T$. Conversely, if $T^*=-T$, one has $(Tx\mid x)=-(x\mid Tx)=-(Tx\mid x)$, so $(Tx\mid x)=0$ for every $x\in H$, that is, $v(T)=0$.

(b). For $x,y\in S_H$, as $T$ is selfadjoint, we have by polarization,
\begin{align*}
4|(Tx\mid y)| &= \bigl|\bigl(T(x+y)\mid x+y\bigr) - \bigl(T(x-y)\mid x-y\bigr)\bigr| \\
& \leq v(T)\bigl(\|x+y\|^2 + \|x-y\|^2 \bigr)\\ &=v(T)\bigl(2\|x\|^2 + 2\|y\|^2 \bigr)=4v(T).
\end{align*}
Taking supremum on $x,y\in S_H$, we get $\|T\|\leq v(T)$.
\end{proof}

\begin{proof}[Proof of Theorem~\ref{thrm-Hilbert}]
Fix $T\in \mathcal{L}(H)$. The result will follow from the following two claims.

\noindent\emph{Claim 1.\ } $\|T+\mathcal{Z}(H)\|=\left\|\dfrac{T+T^*}{2}\right\|$.

We have
$$
T=\dfrac{T+T^*}{2} + \dfrac{T-T^*}{2} \ \text{~and~} \ \dfrac{T-T^*}{2} \in \mathcal{Z}(H)
$$
 by Lemma~\ref{lemma-Hilbert}.a. This shows that $\|T+\mathcal{Z}(H)\|\leq \left\|\dfrac{T+T^*}{2}\right\|$. On the other hand, for every $S\in \mathcal{Z}(H)$, we have that $S+S^*=0$, so
\begin{align*}
\left\|\dfrac{T+T^*}{2}\right\| &\leq \left\|\dfrac{T+T^*}{2}+ \dfrac{S+S^*}{2}\right\|=\left\|\dfrac{(T+S) + (T+S)^*}{2}\right\|\leq \|T+S\|.
\end{align*}
Taking infimum on $S\in \mathcal{Z}(H)$, we get $\left\|\dfrac{T+T^*}{2}\right\|\leq \|T+\mathcal{Z}(H)\|$, as desired.

\noindent\emph{Claim 2.\ } $v(T)=\left\|\dfrac{T+T^*}{2}\right\|$.

We use again that
$$
T=\dfrac{T+T^*}{2} + \dfrac{T-T^*}{2} \ \text{~and~} \ \dfrac{T-T^*}{2} \in \mathcal{Z}(H)
$$ by Lemma~\ref{lemma-Hilbert}.a, to get that
$$
V(T)=V\left(\dfrac{T+T^*}{2}\right), \qquad \text{so} \quad v(T)=v\left(\dfrac{T+T^*}{2}\right).
$$
Since $\dfrac{T+T^*}{2}$ is selfadjoint, Lemma~\ref{lemma-Hilbert}.b gives us that
\begin{equation*}
v(T)=v\left(\dfrac{T+T^*}{2}\right)=\left\|\dfrac{T+T^*}{2}\right\|.\qedhere
\end{equation*}
\end{proof}

More examples will be given in section~\ref{sect-sums-moreexamples}.

\section{An application to the Bishop-Phelps-Bollob\'{a}s property}\label{sect-application}

Our goal here is to give an application of the second numerical index to the theory of numerical radius attaining operators. We first need some definitions.

\begin{definitions} Let $X$ be a Banach space.
\begin{enumerate}
\item[(a)] \textrm{\cite{GuiOle}}\ $X$ is said to have the \emph{Bishop-Phelps-Bollob\'as property for numerical radius} if for every $0<\eps<1$, there exists $\eta(\eps)>0$ such that whenever $T\in \mathcal{L}(X)$ and $(x, x^*)\in \Pi(X)$ satisfy  $v(T)=1$ and $|x^*Tx|>1-\eta(\eps)$, there exit $S\in \mathcal{L}(X)$ and $(y, y^*)\in \Pi(X)$ such that
\[
v(S) = |y^*Sy|=1, \ \ \ \|T-S\|<\eps, \ \ \ \|x-y\|<\eps,\ \ \text{and} \ \ \  \|x^*- y^*\|<\eps.
\]
\item[(b)] \textrm{\cite{KLM-BPBp-nu}}\ $X$ is said to have the \emph{weak-Bishop-Phelps-Bollob\'as property for numerical radius} if for every $0<\eps<1$, there exists $\eta(\eps)>0$ such that whenever $T\in \mathcal{L}(X)$ and $(x, x^*)\in \Pi(X)$ satisfy $v(T)=1$ and $|x^*Tx|>1-\eta(\eps)$, there exit $S\in \mathcal{L}(X)$ and $(y, y^*)\in \Pi(X)$ such that
\[
v(S) = |y^*Sy|, \ \ \ \|T-S\|<\eps, \ \ \ \|x-y\|<\eps,\ \ \text{and} \ \ \  \|x^*- y^*\|<\eps.
\]
\end{enumerate}
\end{definitions}

Notice that the only difference between these two concepts is the normalization of the numerical radius of the operator $S$.

The following result is an extension of \cite[Proposition~6]{KLM-BPBp-nu}, where it is proved under the more restrictive assumption $n(X)>0$.

\begin{theorem}\label{theorem-weak-BPBnu=>BPBnu}
Let $X$ be a Banach space with $n'(X)>0$. Then, the weak-Bishop-Phelps-Bollob\'{a}s property for numerical radius and the Bishop-Phelps-Bollob\'{a}s property for numerical radius are equivalent in $X$.
\end{theorem}

\begin{proof}
One implication is clear. For the converse, assume that for each $0<\eps<1$ we have $\eta(\eps)>0$ satisfying the conditions of the weak-Bishop-Phelps-Bollob\'{a}s property for numerical radius. Fix  $0<\eps<1$. If $T\in \mathcal{L}(X)$ with $v(T)=1$ and $(x, x^*)\in \Pi(X)$ satisfy that
$|x^*Tx|>1-\eta(\eps)$, there exist $S\in \mathcal{L}(X)$ and $(y, y^*)\in \Pi(X)$ such that
\[ v(S) = |y^*Sy|, \ \ \  \|S-T\|<\eps, \ \ \ \|x-y\|<\eps~~~ \text{and}~~~\|x^*-y^*\|<\eps.\]
Observe that
$$
\eps\geq\|S-T\|\geq v(S-T)\geq v(T)-v(S)=1-v(S)
$$
so $v(S)\geq 1-\eps>0$ and we may and do define $S_1 = \frac{1}{v(S)}S$. Then we have
\[  1=v(S_1)= |y^*S_1 y|,~~~~\|x-y\|<\eps~~~ \text{and}~~~\|x^*-y^*\|<\eps.\]
Finally, we can write
\begin{align*}
 \bigl\|\bigl(S_1+\mathcal{Z}(X)\bigr) -\bigl(T+\mathcal{Z}(X)\bigr)\bigr\| &\leq \left\|\frac{1}{v(S)} \bigl(S+\mathcal{Z}(X)\bigr) - \bigl(S+\mathcal{Z}(X)\bigr)\right\| + \bigl\|\bigl(S+\mathcal{Z}(X)\bigr)- \bigl(T+\mathcal{Z}(X)\bigr)\bigr\| \\ &\leq \frac{\|S+\mathcal{Z}(X)\|}{v(S)} |v(S) - 1| + \bigl\|\bigl(S+\mathcal{Z}(X)\bigr)- \bigl(T+\mathcal{Z}(X)\bigr)\bigr\|\\
&\leq \frac{1}{n'(X)} |v(S)- v(T)| + \|S-T\|\\
&\leq \left( \frac{1}{n'(X)}+1\right) \|S-T\| < \frac{n'(X)+1}{n'(X)} \eps.
\end{align*}
So there exists $S_2\in \mathcal{Z}(X)$ such that $\|(S_1+S_2) - T\|<\frac{n'(X)+1}{n'(X)} \eps.$
Then we have
\[ 1=v(S_1) = v(S_1+S_2)= |y^*(S_1+S_2) y|.\]
An obvious change of parameters finishes the proof.
\end{proof}

As a consequence, we obtain that these two properties are equivalent for finite-dimensional spaces but, actually, both are always true in this case. This was proved in \cite[Proposition~2]{KLM-BPBp-nu} and the argument given there is very similar to the above one. On the other hand, it is proved in \cite[Proposition~4]{KLM-BPBp-nu} that a Banach space $X$ which is both uniformly convex and uniformly smooth has the weak-Bishop-Phelps-Bollob\'{a}s property for numerical radius. When $n'(X)>0$, Theorem~\ref{theorem-weak-BPBnu=>BPBnu} gives that $X$ actually has the Bishop-Phelps-Bollob\'{a}s property for numerical radius. An interesting case, which is not covered by \cite[Corollary~7]{KLM-BPBp-nu}, is the following. It is just a consequence of the above discussion and Theorem~\ref{thrm-Hilbert}.

\begin{corollary}
Let $X$ be a Hilbert space. Then $X$ has the Bishop-Phelps-Bollob\'{a}s property for the numerical radius.
\end{corollary}

\section{Direct sums of Banach spaces and more examples}\label{sect-sums-moreexamples}

We present now an inequality for the second numerical index of absolute sums of Banach spaces.

\begin{prop}\label{prop:absolute-summand}
Let $X$ be a Banach space and $Y$, $W$ closed subspaces of $X$ such that $X=Y \oplus_a W$, where $\oplus_a$ is an absolute sum different from the $\ell_2$-sum. Then,
$$
n'(X)\leq \min\{n'(Y),n'(W)\}.
$$
\end{prop}

We need the following lemma, which is based on Lemma~3.3 of \cite{ChicaMartinMeri-QJM2014} and uses Lemma~\ref{lemma:lemma5Paya82}.

\begin{lemma}\label{lemma-abs-sum}
Let $X$ be a Banach space and $Y$, $W$ closed subspaces of $X$ such that $X=Y \oplus_a W$, where $\oplus_a$ is an absolute sum different from the $\ell_2$-sum. Then, given an operator $T\in \mathcal{L}(Y)$, the operator $\widetilde{T}\in \mathcal{L}(X)$ defined by
$$
\widetilde{T}(y+w)=Ty \qquad (y\in Y, w\in W),
$$
satisfies $\|\widetilde{T}\|=\|T\|$, $v(\widetilde{T})=v(T)$, and $\|\widetilde{T}+\mathcal{Z}(X)\|=\|T+\mathcal{Z}(Y)\|$.
\end{lemma}

\begin{proof} Fix $T\in \mathcal{L}(Y)$. That $\|\widetilde{T}\|=\|T\|$ and $v(\widetilde{T})=v(T)$ is proved (under weaker assumptions) in \cite[Lemma~3.3]{ChicaMartinMeri-QJM2014}. Let us prove now that
$\|\widetilde{T}+\mathcal{Z}(X)\|\leq \|T+\mathcal{Z}(Y)\|$. Indeed, for $S\in \mathcal{Z}(Y)$, we have $v(\widetilde{S})=v(S)=0$, so
$$
\|\widetilde{T}+\mathcal{Z}(X)\| \leq \|\widetilde{T} - \widetilde{S}\|=\|\widetilde{T-S}\|=\|T-S\|.
$$
Taking infimum on $S\in \mathcal{Z}(Y)$, we get the claim.

For the reversed inequality, let us fix $\widehat{S}\in \mathcal{Z}(X)$ and use that $\oplus_a$ is not the $\ell_2$-sum and Lemma~\ref{lemma:lemma5Paya82} to get that
$$
\widehat{S}=\begin{pmatrix} S_1 & 0 \\ 0 & S_2 \end{pmatrix}
$$
where $S_1 \in \mathcal{Z}(Y)$ and $S_2\in \mathcal{Z}(W)$. Now,  one has
\begin{align*}
\|\widetilde{T}-\widehat{S}\|= \left\|\begin{pmatrix} T-S_1 & 0 \\ 0 & S_2 \end{pmatrix}\right\| \geq \|T-S_1\|\geq \|T+\mathcal{Z}(Y)\|,
\end{align*}
where the first inequality is true since we are dealing with an absolute sum. Finally, we just have to take infimum on $\widehat{S}\in \mathcal{Z}(X)$ to get the desired inequality, finishing the proof.
\end{proof}

\begin{proof}[Proof of Proposition~\ref{prop:absolute-summand}]
Fix $T\in \mathcal{L}(Y)$ with $\|T+\mathcal{Z}(Y)\|\neq 0$. By Lemma~\ref{lemma-abs-sum}, there is an operator $\widetilde{T}\in \mathcal{L}(X)$ with $\|\widetilde{T}+\mathcal{Z}(X)\|=\|T+\mathcal{Z}(Y)\|$ and such that $v(\widetilde{T})=v(T)$. Then,
$$
n'(X)\leq \frac{v(\widetilde{T})}{\|\widetilde{T}+\mathcal{Z}(X)\|}= \frac{v(T)}{\|T+\mathcal{Z}(Y)\|}.
$$
Taking infimum with $T\in \mathcal{L}(Y)$ with $\|T+\mathcal{Z}(Y)\|\neq 0$, we get $n'(X)\leq n'(Y)$.
\end{proof}

As a consequence of Proposition~\ref{prop:absolute-summand}, we obtain the following result for $c_0$ and $\ell_p$-sums.

\begin{corollary}
\label{suma} Let $\{X_\lambda\,:\, \lambda \in \Lambda\}$ be a
family of Banach spaces. Then
\begin{equation*}
n'\Bigl(\left[\oplus_{\lambda\in\Lambda} X_\lambda\right]_{c_0}
\Bigr)\leq \inf_{\lambda} \,n'(X_\lambda), \qquad \text{and} \qquad n'\Bigl(\left[\oplus_{\lambda\in\Lambda}
X_\lambda\right]_{\ell_p} \Bigr)\leq \inf_{\lambda} \,n'(X_\lambda) \ \text{ for $1\leq p \leq \infty$, $p\neq 2$}.
\end{equation*}
\end{corollary}

\begin{proof}
If we write $X=\left[\oplus_{\lambda\in\Lambda} X_\lambda\right]_{\ell_p}$, then for every $\lambda\in \Lambda$, $X=X_\lambda \oplus_p Y$ for suitable space $Y$. Then, the result follows from Proposition~\ref{prop:absolute-summand}. The case of $c_0$ is absolutely analogous.
\end{proof}

We will extend this corollary to more general absolute sums of Banach spaces in section~\ref{sec:absolute-sums}.

We use now Proposition~\ref{prop:absolute-summand} to give some interesting examples.

\begin{example}
{\slshape The second numerical index is not continuous with respect to the Banach-Mazur distance, even in the set of Banach spaces with numerical index zero.}

Indeed, for $1<p<\infty$, let $X_p=\ell_p^2 \oplus_p \ell_2^2$. First observe that the mapping $p\longmapsto X_p$ is continuous with respect to the Banach-Mazur distance. Next, we have that $n(X_p)=0$ for every $p$ as $n(X_p)\leq n(\ell_2^2)=0$ by Proposition~\ref{prop-suma-oldMarPay2000}. Also, $n'(X_p)\leq n'(\ell_p^2)=n(\ell_p^2)$ by Proposition~\ref{prop:absolute-summand}, so $\lim\limits_{p\to 2} n'(X_p)=0$ (see Example~\ref{example:old-values-lp2}). On the other hand, $n'(X_2)=n'(\ell_2^4)=1$ by Theorem~\ref{thrm-Hilbert}.
\end{example}

Observe that in the previous example the Lie algebra of the limit space differs from that of the limiting spaces. In the next result we show that the second numerical index is continuous with respect to the Banach-Mazur distance when we restrict to spaces with the same Lie algebra. Its proof follows the lines of Proposition~2 in \cite{F-M-P} from where we also borrow some notation. Given a Banach space $X$ with Lie algebra $\mathcal{Z}(X)$ we denote by $\mathcal{E}'(X)$ the set of
all equivalent norms on the Banach space X whose Lie algebra coincides with $\mathcal{Z}(X)$. This is a metric
space when provided with the distance given by
$$
d(p, q) = \log\left( \min \{k\geq 1\, : \, p\leq kq,\, q\leq kp\}\right) \qquad  (p, q \in \mathcal{E}'(X)).
$$
If $p\in \mathcal{E}'(X)$ and $T \in \mathcal{L}(X)$, we write $v_p(T)$ for the numerical radius of $T$ in the
space $(X, p)$, we also use the symbol $p$ to denote the associated operator norm and so $p(T+\mathcal{Z}(X))=\inf_{S\in \mathcal{Z}(X)}p(T-S)$. Finally, $n'(X, p)$ is the second numerical index of the Banach space $(X, p)$.

\begin{prop}\label{prop:continuity-B-M-distance-same-Lie-algebra}
Given a Banach space X, the mapping $p\longmapsto n'(X, p)$ from $\mathcal{E}'(X)$
to $\R$ is continuous.
\end{prop}

\begin{proof}
Fix $p_0\in \mathcal{E}'(X)$ and $r>0$, let $B=\{p\in \mathcal{E}'(X) \, : \, d(p,p_0)< \log(1+r)\} $, and
$$
\mathcal{S}=\{T\in \mathcal{L}(X) \, : \, p_0(T+\mathcal{Z}(X))=1,\, p_0(T)\leq 2\}.
$$
Our first aim is to show that the mapping
$$
(p,T)\longmapsto p(T+\mathcal{Z}(X))
$$
is uniformly continuous on $B\times \mathcal{S}$. To do so, take $T_1,T_2\in \mathcal{S}$ and, fixed $0<\delta<1$, take $p,q\in B$ satisfying $d(p,q)\leq \log(1+\delta)$ which gives $p\leq (1+\delta)q$ and $q\leq (1+\delta)p$. Observe that
$$
p(T_1+\mathcal{Z}(X))-p(T_2+\mathcal{Z}(X))\leq p(T_1-T_2)\leq (1+r)p_0(T_1-T_2).
$$
Therefore, we can write
\begin{align*}
p(T_1+\mathcal{Z}(X))-q(T_2+\mathcal{Z}(X)) & = p(T_1+\mathcal{Z}(X))-p(T_2+\mathcal{Z}(X))+p(T_2+\mathcal{Z}(X))-q(T_2+\mathcal{Z}(X))\\
& \leq  (1+r)p_0(T_1-T_2)+p(T_2+\mathcal{Z}(X))-q(T_2+\mathcal{Z}(X)).
\end{align*}
Now, given $\eps>0$, take $S\in \mathcal{Z}(X)$ satisfying $q(T_2-S)-q(T_2+\mathcal{Z}(X))\leq \eps$ and
observe that
\begin{align*}
p(T_2+\mathcal{Z}(X))-q(T_2+\mathcal{Z}(X))&\leq p(T_2-S)-q(T_2+\mathcal{Z}(X)) \leq p(T_2-S)-q(T_2-S)+\eps \\
&\leq (1+\delta)q(T_2-S)-q(T_2-S)+\eps\\
& =\delta q(T_2-S)+\eps \leq \delta q(T_2+\mathcal{Z}(X))+(1+\delta)\eps \\
& \leq \delta (1+r) p_0(T_2+\mathcal{Z}(X))+(1+\delta)\eps=\delta (1+r) +(1+\delta)\eps.
\end{align*}
So we can continue the above estimation as follows:
\begin{align*}
p(T_1+\mathcal{Z}(X))-q(T_2+\mathcal{Z}(X))&\leq  (1+r)p_0(T_1-T_2)+p(T_2+\mathcal{Z}(X))-q(T_2+\mathcal{Z}(X))\\
&\leq (1+r)p_0(T_1-T_2)+\delta (1+r)+(1+\delta)\eps
\end{align*}
and the arbitrariness of $\eps$ gives
$$
p(T_1+\mathcal{Z}(X))-q(T_2+\mathcal{Z}(X))\leq(1+r)p_0(T_1-T_2)+\delta (1+r).
$$
Finally, exchanging the roles of $p$ and $q$ we obtain
$$
\Big|p(T_1+\mathcal{Z}(X))-q(T_2+\mathcal{Z}(X))\Big|\leq (1+r)p_0(T_1-T_2)+\delta (1+r)
$$
which gives the uniform continuity of the mapping $(p,T)\longmapsto p(T+\mathcal{Z}(X))$ on $B\times \mathcal{S}$.

On the other hand, it is observed in the proof of \cite[Proposition~2]{F-M-P} that the mapping $(p,T)\longmapsto v_p(T)$ is uniformly continuous on bounded sets, so it is uniformly continuous on $B\times \mathcal{S}$. Therefore, using the fact that
$$
\inf\{p(T+\mathcal{Z}(X)) \, : \, p\in B,\, T\in \mathcal{S}\}\geq \frac{1}{1+r}>0
$$
we deduce that the mapping $\Psi: B\times \mathcal{S} \longrightarrow \R$ given by
$$
\Psi(p,T)=\frac{v_p(T)}{p(T+\mathcal{Z}(X))} \qquad (p\in B, T\in\mathcal{S})
$$
is uniformly continuous. This implies that the mapping
$$
p\longmapsto \inf\{\Psi(p,T) \ : \ T\in \mathcal{S}\}= n'(X,p)
$$
is continuous on $B$.
\end{proof}

For the classical numerical index, the inequality in Proposition~\ref{prop:absolute-summand} is an equality when the absolute sum is the $\ell_1$-sum or the $\ell_\infty$-sum. The next example shows that this is not the case for the second numerical index.

\begin{example}\label{example:Hilbert-sum-infty-Hilbert}
{\slshape Let $H$ be a Hilbert space with $\dim(H)\geq 2$ and $W$ be a nontrivial Banach space. Suppose that  $X=H\oplus_\infty W$ or $X=H\oplus_1 W$. Then $n'(X)\leq \frac{\sqrt{3}}{2}<1$.}
\end{example}

\begin{proof} Fix two orthogonal elements $y_1, y_2\in S_H$ and $w_1\in S_W$. Choose $w_1^*\in S_{W}$ such that $w_1^*(w_1)=1$. We suppose first that $X=H\oplus_\infty W$ and we consider the operator $T_1\in \mathcal{L}(X)$ given by
$$
T_1(y,w)= \big((y_1\mid y ) y_1 +  \sqrt{2}w_1^*(w) y_2,0\big) \qquad \big( (y,w) \in X\big).
$$
We start showing that $v(T_1)\leq \frac{3}{2}$. Indeed, given $\big((y,w),(y^*,w^*)\big)\in \Pi(X)$, we have that
$$
\max\{\|y\|,\|w\|\}=1,\quad \|y^*\|+\|w^*\|=1,\quad \text{and}\quad
1=y^*(y)+w^*(w)=\|y^*\|\|y\|+\|w^*\|\|w\|.
$$
Therefore, $y^*(y)=\|y^*\|\|y\|$ and $w^*(w)=\|w^*\|\|w\|$.
Since $H$ is a Hilbert space and $y^*(y)=\|y^*\|\|y\|$, we get that $\|y\|y^*=y\|y^*\|$. If $\|y\|<1$, then necessarily $\|w\|=1$, so $\|w^*\|=1$ and $\|y^*\|=0$ (otherwise we would get $\|y^*\|\|y\|+\|w^*\|\|w\|<1$). Therefore, in this case, we have that
$$
\left|(y^*,w^*)T_1(y,w)\right|=0.
$$
If, otherwise, $\|y\|=1$, we get that $y^*=y\|y^*\|$ and we can write
\begin{align*}
\left|(y^*,w^*)T_1(y,w)\right|&=\left|( y_1 \mid y) y^*(y_1)+ \sqrt{2}w_1^*(w) y^*(y_2)\right|\\
&= \big|\|y^*\|( y_1\mid y)\,(y \mid y_1) + \sqrt{2}\|y^*\|w_1^*(w) (y \mid y_2)\big|\\
&\leq \|y^*\|\big(|(y_1\mid y)|^2 + \sqrt{2}|(y \mid y_2)|\big)\leq |(y_1\mid y)|^2+\sqrt{2}|( y\mid y_2)|.
\end{align*}
This, together with the inequality $1=\|y\|^2\geq |(y_1\mid y)|^2+ |(y\mid y_2)|^2$, tells us that
$$
\left|(y^*,w^*)T_1(y,w)\right|\leq 1-|(y\mid y_2))|^2 + \sqrt{2}|(y\mid y_2) |\leq \underset{t\in[0,1]}{\max}1-t^2+\sqrt{2}t= \frac{3}{2}.
$$

We prove next that $\|T_1 + S\|\geq \sqrt{3}$ for every $S\in \mathcal{Z}(X)$. Indeed, fixed $S\in \mathcal{Z}(X)$, by Lemma~\ref{lemma:lemma5Paya82} there are $S_1\in \mathcal{Z}(H)$ and $S_2\in \mathcal{Z}(W)$ such that $S(y,w)=(S_1(y), S_2(w))$ for every $(y,w)\in X$. Besides, we fix $\theta\in \{-1,1\}$ satisfying $\theta( y_2\mid  S_1(y_1)) \geq 0$. Finally, observe that $(y_1\mid S_1(y_1))=0$ since $v(S_1)=0$ and, therefore, we can write
\begin{align*}
\|T_1+S\|&\geq\|[T_1+S](y_1,\theta w_1)\|=\|(y_1+\theta\sqrt{2} y_2,0)+(S_1(y_1),\theta S_2(w_1))\|\\
&\geq \|y_1+\sqrt{2}\theta y_2+S_1(y_1)\|=\sqrt{\bigl(y_1+\sqrt{2}\theta y_2+S_1(y_1) \mid y_1+\sqrt{2}\theta y_2+S_1(y_1)\bigr)}\\
&=\sqrt{3+ \|S_1(y_1)\|^2 + 2\sqrt{2}\theta \bigl( y_2\mid S_1(y_1)\bigr) }\geq \sqrt{3}.
\end{align*}
Taking infimum in $S\in \mathcal{Z}(X)$, we get $\|T_1+\mathcal{Z}(X)\|\geq\sqrt{3}$ which finishes the proof for $X=H\oplus_\infty W$.

The proof for $X=H\oplus_1 W$ is somehow dual to the above one. In fact, let $T_2\in \mathcal{L}(X)$ be given by
$$
T_2(y,w)= \big(( y_1\mid y ) y_1, \sqrt{2} (y_2 \mid y ) w_1\big) \qquad \big( (y,w) \in X\big).
$$
We start showing that $v(T_2)\leq \frac{3}{2}$. Indeed, given $\big((y,w),(y^*,w^*)\big)\in \Pi(X)$, we have that
$$
\|y\|+\|w\|=1, \quad \max\{\|y^*\|,\|w^*\|\}=1,\quad \text{and}\quad
1=y^*(y)+w^*(w)=\|y^*\|\|y\|+\|w^*\|\|w\|.
$$
Therefore, $y^*(y)=\|y^*\|\|y\|$ and $w^*(w)=\|w^*\|\|w\|$.
Since $H$ is a Hilbert space and $y^*(y)=\|y^*\|\|y\|$, we get that $\|y\|y^*=y\|y^*\|$. If $\|y\|=0$, then
$$
\bigl|(y^*,w^*)T_2(y,w)\bigr|=0.
$$
If, otherwise, $\|y\|\neq 0$, we can write
\begin{align*}
\bigl|(y^*,w^*)T_2(y,w)\bigr|&=\left|( y_1\mid y) y^*(y_1)+ \sqrt{2}w^*(w_1) ( y_2\mid y)\right|\\
&= \left|( y_1\mid y) \Big( \|y^*\|\frac{y}{\|y\|}\mid y_1\Big)+\sqrt{2}w^*(w_1) ( y_2\mid y)\right|\\
&\leq |( y_1\mid y)|\left|\Big(\frac{y}{\|y\|}\mid y_1\Big)\right|+\sqrt{2}|( y\mid y_2)|\leq \left|\Big( y_1\mid \frac{y}{\|y\|}\Big)\right|^2+\sqrt{2}\left|\Big( \frac{y}{\|y\|}\mid y_2\Big)\right|.
\end{align*}
This, together with the inequality $1=\big\|\frac{y}{\|y\|}\big\|^2\geq \left|\Big( y_1\mid \frac{y}{\|y\|}\Big)\right|^2+\left|\Big( \frac{y}{\|y\|}\mid y_2\Big)\right|^2$, tells us that
$$
\left|(y^*,w^*)T_2(y,w)\right|\leq 1-\left|\Big( y_2\mid \frac{y}{\|y\|}\Big)\right|^2+\sqrt{2}\left|\Big( \frac{y}{\|y\|}\mid y_2\Big)\right|\leq \frac{3}{2}.
$$
Let us check that $\|T_2 +\mathcal{Z}(X)\|\geq \sqrt{3}$. Fixed $S\in \mathcal{Z}(X)$, by Lemma~\ref{lemma:lemma5Paya82} there are $S_1\in \mathcal{Z}(H)$ and $S_2\in \mathcal{Z}(W)$ such that $S(y,w)=(S_1(y), S_2(w))$ for every $(y,w)\in X$. We fix $\theta\in \{-1,1\}$ satisfying $\theta( y_1\mid  S_1(y_2)) \geq 0$. Besides, observe that $( y_1\mid  S_1(y_1))=0$ since $v(S_1)=0$ and, therefore, we can write
\begin{align*}
\|T_2+S\|&\geq\left\|[T_2+S]\left(\frac{1}{\sqrt{3}}(y_1+\sqrt{2}\theta y_2),0\right)\right\| =\frac{1}{\sqrt{3}}\bigl\|(y_1,2\theta w_1)
+(S_1(y_1+\sqrt{2}\theta y_2),0)\bigr\|\\
&=\frac{1}{\sqrt{3}}\bigl\|y_1+S_1(y_1+\sqrt{2}\theta y_2)\bigr\|+\frac{2}{\sqrt{3}}\\
&=\frac{1}{\sqrt{3}}\sqrt{1+ \|S_1(y_1+\sqrt{2}\theta y_2)\|^2 + 2\sqrt{2}\theta \bigl( y_1\mid  S_1(y_2)\bigr) }+\frac{2}{\sqrt{3}}\geq \frac{1}{\sqrt{3}}+\frac{2}{\sqrt{3}}=\sqrt{3}.
\end{align*}
Taking infimum in $S\in \mathcal{Z}(X)$, we get $\|T_2+\mathcal{Z}(X)\|\geq\sqrt{3}$ which finishes the proof.
\end{proof}

Even though the above example shows that equality in Proposition~\ref{prop:absolute-summand} for the $\ell_\infty$-sum and the $\ell_1$-sum is not always possible, the next one provides us with a lower bound.

\begin{proposition}\label{prop:lower}
Let $X_1$, $X_2$ be Banach spaces and write $X=X_1\oplus_\infty X_2$ or $X=X_1\oplus_1 X_2$.
\begin{enumerate}
\item[(a)] If $n(X_1)>0$ and $n(X_2)>0$, then $n'(X)=\min \left\{n(X_1),n(X_2)\right\}$.
\item[(b)] If $n(X_1)>0$ and $n(X_2)=0$, then $n'(X)\geq \min \left\{n(X_1), \frac{n'(X_2)}{n'(X_2)+1}\right\}$.
\item[(c)] If $n(X_1)=0$ and $n(X_2)= 0$, then
\[
n'(X)\ge \min\left\{ \frac{n'(X_1)}{n'(X_1)+1}, \frac{n'(X_2)}{n'(X_2)+1}\right\}\,.
\]
\end{enumerate}
\end{proposition}

\begin{proof} (a) is given by Proposition~\ref{prop-suma-oldMarPay2000}.

We suppose that $X=X_1\oplus_\infty X_2$ and prove (c).
We may assume that $n'(X_1)>0$ and $n'(X_2)>0$, otherwise there is nothing to prove.
Every $T\in \mathcal{L}(X)$ can be written as follows:
\[
T(x,y) = (T_{11}(x)+T_{12}(y), T_{21}(x)+T_{22}(y))\qquad \bigl((x,y)\in X\bigr),
\]
where $T_{ij}\in \mathcal{L}(X_j, X_i)$ for all $i,j=1,2$. By Lemma~\ref{lemma:lemma5Paya82}, every operator with numerical radius zero commutes with the projections and, therefore, we get
\begin{align*}
\|T+\mathcal{Z}(X) \| &= \underset{i=1,2}{\inf_{S_i\in \mathcal{Z}(X_i)}} \max\left\{ \sup_{\|x_1\|=\|x_2\|=1} \|T_{11}(x_1)+S_1(x_1)+T_{12}(x_2)\|,\right. \\ & \phantom{\inf_{S_i\in \mathcal{Z}(X_i), i=1,2}\qquad\qquad\quad} \left. \sup_{\|x_1\|=\|x_2\|=1} \|T_{21}(x_1)+S_2(x_2)+T_{22}(x_2)\|  \right\}.
\end{align*}
Hence
\begin{align*}
 \|T+\mathcal{Z}(X) \| &\leq \underset{i=1,2}{\inf_{S_i\in \mathcal{Z}(X_i)}} \max\{  \|T_{11}+S_1\|+\|T_{12}\|, \|T_{21}\|+\|S_2+T_{22}\|  \}\\
 &=\max\{  \|T_{11} +\mathcal{Z}(X_1)\|+\|T_{12}\|, \|T_{21}\|+\|T_{22} +\mathcal{Z}(X_2)\|  \}.
\end{align*}
Besides, we claim that
\[
v(T)\ge \max\{ v(T_{11}), \|T_{12}\|, v(T_{22}), \|T_{21}\|\}.
\]
Indeed, we note that for $((x_1, x_2),(x_1^*, x_2^*))\in \Pi(X)$,
\[
v(T) \ge |x_1^*T_{11}(x_1) + x_1^*T_{12}(x_2)+ x_2^*T_{21}(x_1)+x_2^*T_{22}(x_2)|.
\]
Given $\eps >0$, we take $x_2^*=0$ and $(x_1, x_1^*)\in \Pi(X_1)$ satisfying that $|x^*_1T_{11}(x_1)|>v(T_{11})-\eps$, we take $y_2\in S_{X_2}$ and $\theta\in\{-1,1\}$ such that $x^*_1(T_{11}(x_1))x_1^*(T_{12}(y_2))\theta\geq 0$, and we write $x_2=\theta y_2$. Then $( (x_1, x_2), (x_1^*, x_2^*) )\in \Pi(X)$ and, therefore,
\begin{align*}
v(T) &\ge|x^*_1T_{11}(x_1) + x_1^*T_{12}(x_2)|= |x^*_1T_{11}(x_1) + x_1^*T_{12}(y_2)\theta|\ge |x^*_1T_{11}(x_1)|>v(T_{11})-\eps
\end{align*}
and the arbitrariness of $\eps$ gives $v(T)\geq v(T_{11})$.

We turn now to prove $v(T)\geq \|T_{12}\|$. Observe that we can assume that $\|T_{12}\|>0$. Then, given $0<\eps<\|T_{12}\|$ we take this time $x_2^*=0$, $y_2\in S_{X_2}$ satisfying that $\|T_{12}(y_2)\|>\|T_{12}\|-\eps$, we take $x_1^*\in S_{X_1^*}$ such that $x_1^*(T_{12}(y_2))=\|T_{12}(y_2)\|$, we write $x_1=\frac{T_{12}(y_2)}{\|T_{12}(y_2)\|}$ and we take $\theta\in\{-1,1\}$ such that $x^*_1(T_{11}(x_1))x_1^*(T_{12}(y_2))\theta\geq 0$, and we write $x_2=\theta y_2$. Then, we have
\begin{align*}
v(T) &\ge|x^*_1T_{11}(x_1) + x_1^*T_{12}(x_2)|= |x^*_1T_{11}(x_1) + x_1^*T_{12}(y_2)\theta|\ge |x^*_1T_{12}(y_2)|>\|T_{12}\|-\eps
\end{align*}
and the arbitrariness of $\eps$ gives $v(T)\geq \|T_{12}\|$. Symmetric arguments show that $v(T)\ge v(T_{22})$ and $v(T)\geq\|T_{21}\|$.

Now we are ready to finish the proof. Suppose that $\|T+\mathcal{Z}(X)\|=1$ and we first assume that $\|T_{11}+\mathcal{Z}(X_1)\|+\|T_{12}\|\geq1$. If $v(T_{11})\ge \frac{n'(X_1)}{n'(X_1)+1}$, then $v(T)\ge \frac{n'(X_1)}{n'(X_1)+1}$. Otherwise, we have $v(T_{11})\leq \frac{n'(X_1)}{n'(X_1)+1}$ and we can estimate
\[
1\leq\|T_{11}+\mathcal{Z}(X_1)\|+\|T_{12}\| \leq \frac{v(T_{11})}{n'(X_1)} + \|T_{12}\|\leq \frac{1}{n'(X_1)+1}+\|T_{12}\|.
\]
So $\|T_{12}\|\ge \frac{n'(X_1)}{n'(X_1)+1}$ and, therefore
\[
v(T) \ge \|T_{12}\| \ge \frac{n'(X_1)}{n'(X_1)+1}.
\]
On the other hand, if we assume $\|T_{22}+\mathcal{Z}(X_2)\|+\|T_{21}\|\geq1$, then an analogous argument gives
\[
v(T) \ge \frac{n'(X_2)}{n'(X_2)+1}
\]
which completes the proof of (c).

To prove (b), one can follow the above proof with some changes. Indeed, consider the operator $T_1\in \mathcal{L}(X,X_1)$ given by $T_1(x_1,x_2)=T_{11}(x_1)+T_{12}(x_2)$ for $x_1\in X_1$ and $x_2\in X_2$, and observe that
$$
\|T+\mathcal{Z}(X)\|\leq \max\{\|T_1\|, \|T_{21}\|+\|T_{22}+\mathcal{Z}(X_2)\|\}
$$
since $n(X_1)>0$. If $\|T_{21}\|+\|T_{22}+\mathcal{Z}(X_2)\|\geq \|T+\mathcal{Z}(X)\|$, the above proof gives
$$
\frac{v(T)}{\|T+\mathcal{Z}(X)\|}\geq \frac{n'(X_2)}{n'(X_2)+1}\,.
$$
If  $\|T_1\|\geq \|T+\mathcal{Z}(X)\|$ the proof of \cite[Proposition~1]{MarPay} tells us that
$$
\frac{v(T)}{\|T+\mathcal{Z}(X)\|}\geq \frac{v(T)}{\|T_1\|}\geq n(X_1).
$$

The proof for $X=X_1\oplus_1 X_2$ is, in some sense, the dual of the above one. We just point out the main differences. To check  $(c)$, observe that this time, for $T\in \mathcal{L}(X)$, we have
\begin{align*}
\|T+\mathcal{Z}(X) \| &=\underset{i=1,2}{\inf_{S_i\in \mathcal{Z}(X_i)}}   \sup_{\|x_1\|+\|x_2\|=1} \|T_{11}(x_1)+S_1(x_1)+T_{12}(x_2)\|+\|T_{21}(x_1)+S_2(x_2)+T_{22}(x_2)\|\\
&\leq \underset{i=1,2}{\inf_{S_i\in \mathcal{Z}(X_i)}}   \sup_{\|x_1\|+\|x_2\|=1} \|T_{11}(x_1)+S_1(x_1)\|+\|T_{21}(x_1)\|+\|T_{12}(x_2)\|+\|S_2(x_2)+T_{22}(x_2)\|\\
&\leq \underset{i=1,2}{\inf_{S_i\in \mathcal{Z}(X_i)}}   \sup_{\|x_1\|+\|x_2\|=1} \|x_1\|\Big(\|T_{11}+S_1\|+\|T_{21}\|\Big)+\|x_2\|\Big(\|T_{12}\|+\|S_2+T_{22}\|\Big)\\
&\leq \underset{i=1,2}{\inf_{S_i\in \mathcal{Z}(X_i)}}  \max\left\{\|T_{11}+S_1\|+\|T_{21}\|,  \|T_{12}\|+\|S_2+T_{22}\|\right\}\\
&= \max\left\{\|T_{11}+\mathcal{Z}(X_1)\|+\|T_{21}\|,  \|T_{12}\|+\|T_{22}+\mathcal{Z}(X_2)\|\right\}.
\end{align*}
To prove the inequality
\[
v(T)\ge \max\{ v(T_{11}), \|T_{12}\|, v(T_{22}), \|T_{21}\|\}
\]
one can proceed as in the previous case, using an appropriate pair $\big((x_1, x_2),(x_1^*, x_2^*)\big)\in \Pi(X)$ with $x_2=0$ (for the inequalities $v(T)\geq v(T_{11})$ and $v(T)\geq \|T_{21}\|$) or $x_1=0$ (for the inequalities $v(T)\geq v(T_{22})$ and $v(T)\geq \|T_{12}\|$). From that point the proof follows exactly the same lines of the above one.

Finally, the proof of (b) is again a combination of the preceding proof and the one of \cite[Proposition~1]{MarPay} for the $\ell_1$-sum.
\end{proof}

A first consequence of the proposition above and Example~\ref{example:Hilbert-sum-infty-Hilbert} is the following example.

\begin{example}\label{example:HoplusinftyH-Hoplus1H}
{\slshape If $H$ is a Hilbert space of dimension greater than or equal to two, then}
$$
\frac{1}{2}\leq n'(H\oplus_\infty H)\leq\frac{\sqrt{3}}{2} \quad \text{and} \quad \frac{1}{2}\leq n'(H\oplus_1 H)\leq\frac{\sqrt{3}}{2}.
$$
\end{example}

More examples can also be deduced from the results of this section. For instance, we may show that the set of values of the second numerical index among Banach spaces with numerical index $0$ contains the interval $[0,1/2]$.

\begin{example}\label{example:four-dim-valuen'}
\slshape For every $\theta \in [0,1/2]$, there is a Banach space $X_\theta$ such that $n(X_\theta)=0$ and $n'(X_\theta)=\theta$. Moreover, for $0<\theta\leq 1/2$, the space $X_\theta$ can be taken to be four dimensional.
\end{example}

\begin{proof}
For $\theta=0$, the result was given in Example~\ref{example-n'=0}, and all the possible examples have to be infinite-dimensional by Proposition~\ref{prop:finite-dimensional-n'>0}. Consider now $0<\theta\leq 1/2$ and pick a two-dimensional Banach space $Y_\theta$ with $n(Y_\theta)=\theta$ (use Proposition~\ref{prop:old-values}). Let $X_\theta$ be the four-dimensional space $Y_\theta \oplus_\infty \ell_2^2$, which satisfies that $n(X_\theta)\leq n(\ell_2^2)=0$ by Proposition~\ref{prop-suma-oldMarPay2000}. By Proposition~\ref{prop:absolute-summand},
$$
n'(X_\theta)\leq n'(Y_\theta)=n(Y_\theta)=\theta.
$$
On the other hand, Proposition~\ref{prop:lower} gives us that $n'(X_\theta)\geq \min\{\theta,1/2\}=\theta$.
\end{proof}

A similar result can be obtained, up to renorming, in every separable or reflexive Banach space of dimension greater than $4$.

\begin{prop}
Let $X$ be a Banach space of dimension greater than or equal to four which admits a long biorthogonal system (for instance, being $X$ separable or being $X$ reflexive). Then, for every $0<\theta\leq 1/2$, there is a Banach space $X_\theta$ isomorphic to $X$ such that $n(X_\theta)=0$ and $n'(X_\theta)=\theta$.
\end{prop}

\begin{proof}
Take a two-dimensional subspace $Y$ of $X$ and write $X=Y\oplus W$ for convenient subspace $W$ of dimension greater than or equal to two. We may use Proposition~\ref{prop-old:values-up-to-renorming} to get a Banach space $W_\theta$ isomorphic to $W$ with $n(W_\theta)=\theta$. Now, consider the space $X_\theta=\ell_2^2 \oplus_\infty W_\theta$, which is isomorphic to $X$. Then, $n(X_\theta)\leq n(\ell_2^2)=0$ by Proposition~\ref{prop-suma-oldMarPay2000}. On the other hand,
$$
n'(X_\theta)\leq n'(W_\theta)=n(W_\theta)=\theta
$$
by Proposition~\ref{prop:absolute-summand}, and Proposition~\ref{prop:lower} gives us that $n'(X_\theta)\geq \theta$.
\end{proof}

The proof above cannot be done in dimension two or three. In the first case, we actually have the following result.

\begin{example}{\slshape Let $X$ be a two dimensional space with $n(X)=0$. Then $n'(X)=1$.}\newline Indeed, it is proved in \cite[Corollary~2.5]{MaMeRo} that $X$ is isometrically isomorphic to a Hilbert space and so Theorem~\ref{thrm-Hilbert} gives the result.
\end{example}

We will give an obstructive result for the set of values of the second numerical index of three-dimensional spaces with numerical index zero in Proposition~\ref{prop:Obstructive-dim-three}.

\section{Duality}\label{sect-duality}
We are interested here in the relationship between the second numerical index of a Banach space and the one of its dual. Recall that for the classical numerical index, it is known that $n(X^*)\leq n(X)$ for every Banach space $X$ and that this inequality may be strict (see Proposition~\ref{prop:old-inequalityduality} and Example~\ref{example:old-duality}).

With respect to the first result, we do not know whether it is always true for the second numerical index, but we have the following sufficient conditions.

\begin{proposition}\label{prop-duality-uniquenesspredual}
Let $X$ be a Banach space. Suppose that $X$ satisfies any of the following conditions:
\begin{enumerate}
\item[(a)] the norm of $X^*$ is Fr\'{e}chet-smooth on a dense set (e.g.\ $X=\ell_\infty$);
\item[(b)] $B_X$ is the closed convex hull of the continuity points of $\,\Id:(B_X,w)\longrightarrow (B_X,\|\cdot\|)$ (in particular, if $X$ has the RNP, $X$ has the CPCP, $X$ is LUR, $X$ has a Kadec norm, $X=X_1\widetilde{\otimes}_\pi X_2$ where $X_1$ and $X_2$ have the RNP, or $X=\mathcal{L}(R)$ where $R$ is reflexive);
\item[(c)] $Y\subset X^*\subset Y^{**}$ for a Banach space $Y$ which does not contain $\ell_1$ (in particular, if $X$ is a dual space with the weak-RNP, i.e.\ $X=Y^*$ for a space $Y$ which does not contain $\ell_1$);
\item[(d)] $X^*\nsupseteq \ell_1$;
\item[(e)] $X$ is isomorphic to a subspace of a separable $L$-embedded space;
\item[(f)] $X$ is the (unique) predual of a von Neumann algebra or, more generally, $X$ is the (unique) predual of a $JBW^*$-triple;
\item[(g)] there is a separable reflexive space $R$ such that $X^*$ is isometrically isomorphic to a weak-star closed linear subspace $L$ of $\mathcal{L}(R)$ whose intersection with the space of compact operators on $R$ is weak-star dense in $L$.
\end{enumerate}
Then $n'(X^*)\leq n'(X)$.
\end{proposition}

Before giving a proof of the proposition, let us recall that the classical inequality $n(X^*)\leq n(X)$ for every Banach space $X$ is an obvious consequence of the fact that $v(T^*)=v(T)$ and $\|T^*\|=\|T\|$ for every $T\in \mathcal{L}(X)$ (see Proposition~\ref{prop:old-inequalityduality}). We do not know whether $\|T^*+\mathcal{Z}(X^*)\|=\|T+\mathcal{Z}(X)\|$ for every $T\in \mathcal{L}(X)$ in general, but this would be true when every element in $\mathcal{Z}(X^*)$ is the transpose of an element in $\mathcal{Z}(X)$. So the idea behind the proof of the above proposition is to find sufficient conditions to ensure that. We need a couple of lemmas which are of independent interest. The first one is just an adaptation of a result of G.~Godefroy \cite{Godefroy}.

\begin{lemma}[\textrm{\cite[Proposition~VII.1 and proof of Corollary~VII.3]{Godefroy}}]\label{lemma:duality-1}
Let $X$ be a Banach space such that there exists a unique norm-one projection $\pi:X^{***}\longrightarrow X^*$ with weak-star closed kernel. Then every $T\in \mathcal{Z}(X^*)$ is the transpose of an element of $\mathcal{Z}(X)$.
\end{lemma}

The second result is immediate.

\begin{lemma}\label{lemma:duality-2}
Let $X$ be a Banach space. If every $T\in \mathcal{Z}(X^*)$ is the transpose of an element of $\mathcal{Z}(X)$, then $\|T^*+\mathcal{Z}(X^*)\|=\|T+\mathcal{Z}(X)\|$ for every $T\in \mathcal{L}(X)$ and, therefore, $n'(X^*)\leq n'(X)$.
\end{lemma}

\begin{proof}[Proof of Proposition~\ref{prop-duality-uniquenesspredual}]
We only need to provide references giving that any of those conditions implies that there is a unique projection from $X^{***}$ onto $X^*$ with weak-star closed kernel, and then apply the above two lemmas.

Indeed, Theorem~II.1, Examples~II.2 and Theorem~II.3 of \cite{Godefroy} give directly the result for (a), (b), (c), and (d) (actually, in those cases there exists a unique norm-one projection from $X^{***}$ onto $X^*$).

(e) \cite[Theorem~3]{Pfitzner} gives that every separable $L$-embedded space satisfies Godefroy-Talagrand property (X) and  \cite[Theorem~V.3]{Godefroy} shows that this latter property implies what we need. Since property (X) is of isomorphic nature and passes to subspaces, we get the result.

(f) For von Neumann algebras, this a classical result of Sakai; for $JBW^*$-triples, it is a result of G.~Horn \cite{Horn}.

(g) It is shown in \cite{Godefroy2014} that, in this case, $X$ has the RNP, and so the result follows from (b).
\end{proof}

\begin{corollary}
Let $X$ be a reflexive space. Then, $n'(X^*)=n'(X)$.
\end{corollary}

We may give another result in this line for $M$-embedded spaces, for which we actually have a little bit more. Recall that a Banach space $X$ is \emph{$M$}-embedded (or it is an \emph{$M$-ideal} in its bidual) if $X^{***}=X^\perp\oplus_1 W$ for some closed subspace $W$.

\begin{proposition}
Let $X$ be an $M$-embedded space. Then $n'(X^{**})\leq n'(X^*)\leq n'(X)$.
\end{proposition}

\begin{proof}
It is shown in \cite[Proposition~III.2.2]{HWW} that every surjective isometry of $X^{**}$ is the bitranspose of a surjective isometry of $X$. As a consequence, we get using the same ideas as in the proof of \cite[Corollary~VII.3]{Godefroy} that every element in $\mathcal{Z}(X^{**})$ is the transpose of an element of $\mathcal{Z}(X^*)$ and that every element of $\mathcal{Z}(X^{*})$ is the transpose of an element of $\mathcal{Z}(X)$. The result now follows from Lemma~\ref{lemma:duality-2}.
\end{proof}

Next, we would like to present an example showing that the inequality $n'(X^*)\leq n'(X)$ can be strict, even when $n(X)=0$.

\begin{example} {\slshape There exists a Banach space $X$ with }
$$
n(X)=n(X^*)=0 \qquad \text{and}\qquad n'(X^*)<n'(X).
$$
\end{example}

\begin{proof}
Let $W$ be a Banach space with $n(W)=1$ and $n(W^*)=1/3$ (see Example~\ref{example:old-duality}) and let $X=\ell_2^2\oplus_\infty W$ which satisfies $n(X)\leq n(\ell_2^2)=0$ by Proposition~\ref{prop:old-absolutesumsinequality}. Now, $n'(X)\geq 1/2$ by Proposition~\ref{prop:lower}, and $n'(X^*)\leq n'(W^*) = n(W^*)=1/3$ by Proposition~\ref{prop:absolute-summand}.
\end{proof}

The next example shows another strong way in which the second numerical index of a space and the one of its dual can be different.

\begin{example}{\slshape
There exists a Banach space $X$ with the following properties:}
\[
n(X)=n'(X)=1, \ \ n(X^*)=n'(X^*)=0, \ \ \mathcal{Z}(X)=\mathcal{Z}(X^*)=\{0\}.
\]
\end{example}

\begin{proof}
Consider a sequence of Banach spaces $X_k$ with $n(X_k)=1$, $n(X_k^*)>0$ and $\lim_k n(X_k^*)=0$ (see Example~\ref{example:old-duality}). Let
$X= \left[\bigoplus_{k\in \N} X_k \right]_{c_0}$. Then $n(X)=1$ (and so $\mathcal{Z}(X)=\{0\}$) by Proposition~\ref{prop-suma-oldMarPay2000}. This proposition also shows that $n(X^*)=0$ and that $\mathcal{Z}(X^*)=\{0\}$. Then, $n'(X^*)=n(X^*)=0$.
\end{proof}

One more example is the following extension of Example~\ref{example:old-duality}.

\begin{example}
{\slshape Given $0\leq \alpha\leq \beta\leq 1/2$, there is a Banach space $X_{\alpha,\beta}$ with $n(X_{\alpha,\beta})=0$ such that}
$$
n'(X_{\alpha,\beta})=\beta \quad \text{ and } \quad n'(X_{\alpha,\beta}^*)=\alpha.
$$
Indeed, Let $Y_{\alpha,\beta}$ be a Banach space such that $n(Y_{\alpha,\beta})=\beta$ and $n(Y_{\alpha,\beta}^*)=\alpha$, and consider $X_{\alpha,\beta}=\ell_2\oplus_\infty Y_{\alpha,\beta}$.
Then, Propositions \ref{prop:absolute-summand} and \ref{prop:lower} show easily the result.
\end{example}

\section{Vector valued function spaces}\label{sec:vector-valued}

We deal now with vector valued function spaces. As $n'(H\oplus_\infty H)<1$ by Example~\ref{example:Hilbert-sum-infty-Hilbert}, one cannot expect to have the same result that the one for the classical numerical index (see Proposition~\ref{prop:old-ck-l1-linfty}), but we will see that at least one inequality can be proved. We start with spaces of continuous functions.

\begin{prop}\label{newindexC}
Let $L$ be a locally compact Hausdorff topological space and $X$ be a Banach space. Then $n'(C_0(L, X))\leq n'(X)$.
\end{prop}

\begin{proof}
Choose a function $\varphi\in C_0(L)$ such that $\varphi(t_0)=1=\|\varphi\|$ for some $t_0\in L$.
Fix $R$ in $\mathcal{L}(X)\setminus \mathcal{Z}(X)$ and define $T\in \mathcal{L}(C_0(L, X))$ by
\[
[T(f)](t) = R(f(t)) \qquad \bigl(t\in K, \ f\in C_0(L, X)\bigr).
\]
We claim that $v(T) \leq v(R)$ and $\|T+\mathcal{Z}(C_0(L,X))\|\geq \|R+\mathcal{Z}(X)\|$. This shows that
\[ n'(C_0(L, X)) \leq \frac{v(T)}{\|T+\mathcal{Z}(C_0(L,X))\|} \leq \frac{v(R)}{\|R+\mathcal{Z}(X)\|}.
\]
Since $R$ is arbitrary, taking infimum on $R$ we get $n'(C(L, X))\leq n'(X)$.

So let us prove the claim. First, consider
\begin{equation}\label{eq:Gamma_en_cdek}
\Gamma = \bigl\{ (f, \delta_t \otimes x^*) \,:\, f\in S_{C_0(L, X)},\, t\in L,\, x^*\in S_{X^*},\, x^*(f(t))=1\bigr\},
\end{equation}
where $[\delta_t\otimes x^*](f) = x^*(f(t))$. Then the numerical radius $v(T)$ can be computed with only those elements in $\Pi(X)$ which are in $\Gamma$ (see Lemma~\ref{lemma:old-projection-dense}) and so
\begin{align*}
v(T) &= \sup\{ |\langle x^*, [Tf](t)\rangle|\,:\, f\in S_{C_0(L, X)},\, t\in L,\, x^*\in S_{X^*},\, x^*(f(t))=1\}\\
&=\sup\{ |\langle x^*,R(f(t))\rangle|\,:\, f\in S_{C_0(L, X)},\, t\in L,\, x^*\in S_{X^*},\, x^*(f(t))=1\}\\
&\leq \sup\{ |x^*(R(x))|\,:\, x\in S_X,\, x^*\in S_{X^*},\, x^*(x)=1\}=v(R).
\end{align*}
To show that $\|R+\mathcal{Z}(X)\|\leq \|T+\mathcal{Z}(C_0(L,X))\|$, given $U\in \mathcal{Z}(C_0(L, X))$, define $Vx = [U(\varphi \otimes x)](t_0)$ for every $x\in X$. Then, $V\in \mathcal{L}(X)$. In fact, $V$ is in $\mathcal{Z}(X)$. Indeed, for each $(x, x^*)\in \Pi(X)$,  $[\delta_{t_0}\otimes x^*](\varphi \otimes x)=1$ and
\[ |x^*V(x)| = \bigl|\langle x^*, U (\varphi\otimes x)(t_0)\rangle\bigr| = |(\delta_{t_0}\otimes x^*) U(\varphi\otimes x)|\leq v(U)=0.\] For each $x\in S_X$, we have
\begin{align*}
\|R+\mathcal{Z}(X)\|\leq \|R+V\| &= \sup_{x\in S_X}\|Rx+Vx\| \\
&= \sup_{x\in S_X}\|Rx + [U(\varphi\otimes x)](t_0)\| \\
&= \sup_{x\in S_X}\|[T(\varphi\otimes x)](t_0) + [U(\varphi\otimes x)](t_0)\|\leq \|T+U\|.
\end{align*}
Hence $\|R+\mathcal{Z}(X)\|\leq \|T+U\|$. Since $U\in \mathcal{Z}(C_0(L, X))$ is arbitrary, we have that $\|R+\mathcal{Z}(X)\|\leq \|T+\mathcal{Z}(C_0(L,X))\|$ and this completes the proof.
\end{proof}

Easy modifications of the above proof can be applied to get results for $C_w(K,X)$ and $C_b(\Omega,X)$.

\begin{remark}
{\slshape Let $K$ be a compact Hausdorff topological space, let $\Omega$ be a completely regular Hausdorff topological space and let $X$ be a Banach space. Then}
$$
n'(C_w(K,X))\leq n'(X) \qquad \text{and} \qquad n'(C_b(\Omega,X))\leq n'(X).
$$
\end{remark}

\begin{proof}
In both cases, the arguments are the same that the one in the proof of Proposition~\ref{newindexC}. Given $R\in \mathcal{L}(X)\setminus \mathcal{Z}(X)$, define $T$ analogously. In the case of $C_w(K,X)$, we just have to replace the set $\Gamma$ in \eqref{eq:Gamma_en_cdek} by
\begin{equation*}
\Gamma_1 = \bigl\{ (f, \delta_t \otimes x^*) \,:\, f\in S_{C_w(K, X)},\, \text{$f$ attains its norm},\, t\in K,\, x^*\in S_{X^*},\, x^*(f(t))=1\bigr\},
\end{equation*}
and observe that $\pi_X(\Gamma_1)$ is dense in $S_{C_w(K,X)}$ by \cite[Lemma~1]{LopezMartinMeri}, so we may use $\Gamma_1$ to compute the numerical radius of $T$. In the case of $C_b(\Omega,X)$, it is straightforward to show that the set of functions in $C_b(\Omega,X)$ which attain their norm is dense in $C_b(\Omega,X)$, and so the set \begin{equation*}
\Gamma_2 = \bigl\{ (f, \delta_t \otimes x^*) \,:\, f\in S_{C_b(\Omega, X)},\, \text{$f$ attains its norm},\, t\in K,\, x^*\in S_{X^*},\, x^*(f(t))=1\bigr\}
\end{equation*}
can be used to compute the numerical radius of $T$. The rest of the proof is the same that the one of Proposition~\ref{newindexC}.
\end{proof}

Another vector-valued function space for which we may give an inequality of the second numerical index is $L_\infty(\mu,X)$.

\begin{prop}\label{prop:n'-Linfty}
Let $(\Omega,\Sigma,\mu)$ a measurable space and let $X$ be Banach space. Then
$$
n'(L_\infty(\mu,X))\leq n'(X).
$$
\end{prop}

For a measurable subset $E$, $x\in X$ and $x^*\in X^*$, we will use the following notations:
\begin{align*}
[\chi_E\otimes x](t)&=x\cdot \chi_E(t), \ \ \ t\in \Omega\\
[\chi_E\otimes x^*](f)& =\langle \chi_E\otimes x^*, f\rangle = \int_E x^*(f(t))\, d\mu(t),  \ \ \  f\in L_1(\mu, X).
\end{align*}

\begin{proof}[Proof of Proposition~\ref{prop:n'-Linfty}]
As in the proof of Proposition~\ref{newindexC}, given $R$ in $\mathcal{L}(X)\setminus \mathcal{Z}(X)$, define $T$ by
$$
[T(f)](t)=R(f(t)) \qquad \bigl(t\in\Omega,\,f\in L_\infty(\mu,X)\bigr).
$$
Now, as shown in \cite[Lemma~2.2]{MarVil}, the numerical radius of $T$ can be computed using the set
\begin{align*}
\Gamma = \Bigl\{ \left(f, \frac{\chi_E }{\mu(E)}\otimes x^*\right) \,:\ &  f\in S_{L_\infty(\mu, X)}, \ E\in \Sigma,\ 0< \mu(E)<\infty, \\
 & f \text{ is constant on } E,\  x^*(f(t)) =\|f\|_\infty=1 \text{ on } E\Bigr\}
\end{align*}
(since $\Gamma$ is a subset of $\Pi(L_\infty(\mu, X))$ and $\pi_X(\Gamma)$ is dense in $S_{L_\infty(\mu, X)}$, so we may use  Lemma~\ref{lemma:old-projection-dense}). Then, repeating the argument given in Proposition~\ref{newindexC}, we get $v(T) \leq v(R)$. Finally, for  $U\in \mathcal{Z}(L_\infty(\mu, X))$,  define $V\in \mathcal{L}(X)$ by
\[
V(x) = \frac{1}{\mu(E_0)}\int_{E_0} \bigl[U( \chi_{\Omega}\otimes x)\bigr](t) \, d\mu(t) \qquad (x\in X),
\]
where $E_0\subset \Omega$ is measurable and $0<\mu(E_0)<\infty$. In fact, $V$ is in $\mathcal{Z}(X)$. Indeed, if $(x, x^*)\in \Pi(X)$, we have
\[ |x^*V(x)| =  \left|\frac{1}{\mu(E_0)}\int_{E_0} x^*U( \chi_{\Omega}\otimes x)(t) \, d\mu(t)\right| = \left|\bigl\langle \frac{\chi_{E_0}}{\mu(E_0)}\otimes x^*, U(\chi_{\Omega}\otimes x)\bigr\rangle \right|\leq v(U)=0.
\]
Since $\bigl[T(\chi_\Omega\otimes x)\bigr](t)=R(x)$ on $\Omega$, we have
\begin{align*}
\|R+\mathcal{Z}(X)\|\leq \|R+V\|&=\sup_{x\in S_X}\|Rx+Vx\| \\
&= \sup_{x\in S_X}\left\|Rx +  \frac{1}{\mu(E_0)}\int_{E_0} U( \chi_{\Omega}\otimes x)(t) \, d\mu(t)\right\| \\
&= \sup_{x\in S_X}\left\|  \frac{1}{\mu(E_0)}\int_{E_0} \bigl([T( \chi_{\Omega}\otimes x)](t)+[U( \chi_{\Omega}\otimes x)](t)\bigr) \, d\mu(t)\right\|\\
&\leq \|T( \chi_{\Omega}\otimes x)+U( \chi_{\Omega}\otimes x)\|_\infty\leq \|T+U\|.
\end{align*}
Therefore, we get $\|R+\mathcal{Z}(X)\|\leq \|T+\mathcal{Z}(L_\infty(\mu,X))\|$ and so $n'(L_\infty(\mu, X))\leq n'(X)$.
\end{proof}

Our next result is for spaces of vector-valued integrable functions.

\begin{prop}
Let $\mu$ be a positive measure. Then $n'(L_1(\mu, X))\leq n'(X)$.
\end{prop}

\begin{proof}
Notice that $L_1(\mu, X)$ is isometrically isomorphic to the $\ell_1$-sum of $L_1(\mu_i, X)$ for suitable finite measures $\mu_i$. So by Proposition~\ref{prop:absolute-summand}, we may and do assume that $\mu$ is a finite measure.

Let $\Gamma$ be the set of all $(f, g)\in L_1(\mu)\times L_\infty(\mu)$ such that $f = \sum_{i=1}^n x_i \chi_{E_i}$, $g = \sum_{i=1}^n x_i^* \chi_{E_i}$ for some $n$, where $x_i$'s are in $X\setminus\{0\}$ and $x_i^*$'s are in $S_{X^*}$, $E_i$ are disjoint measurable subsets of positive finite measure, $\sum_{i=1}^n \|x_i\|\mu(E_i) = 1$ and  $x_i^*(x_i)=\|x_i\|$ for all $1\leq i \leq n$. Since the set of simple functions is dense in $L_1(\mu, X)$, the numerical radius $v(T)$ of $T\in \mathcal{L}(L_1(\mu, X))$ can be computed using only elements in $\Gamma$ by Lemma~\ref{lemma:old-projection-dense}.

Given $R$  in $\mathcal{L}(X)\setminus \mathcal{Z}(X)$, define $T\in \mathcal{L}(L_1(\mu, X))$ by $T(f)(t) = R(f(t))$. Then for $(f, g)\in \Gamma$ with $f = \sum_{i=1}^n x_i \chi_{E_i}$, $g = \sum_{i=1}^n x_i^* \chi_{E_i}$, we have
\[
|\langle g, Tf \rangle| \leq \sum_{i=1}^n |x_i^*R(x_i)| \mu(E_i) \leq \sum_{i=1}^n v(R)\|x_i\|\mu(E_i) = v(R).
\] So $v(T) \leq v(R)$.  Finally, for  $U\in \mathcal{Z}(L_1(\mu, X))$,  define $V\in \mathcal{L}(X)$ by  \[V(x) =\frac{1}{\mu(\Omega)}\int_{\Omega} [U(\chi_\Omega \otimes x)](t)  \, d\mu(t).\] We will check that $V\in \mathcal{Z}(X)$. Indeed, for each $(x, x^*)\in \Pi(X)$, we have
\[ |x^*V(x)| =\left|\frac{1}{\mu(\Omega)}\int_{\Omega} x^*\bigl(U(\chi_\Omega \otimes x)(t)\bigr)  \, d\mu(t) \right|=  \left|\Bigl\langle \frac{\chi_{\Omega}}{\mu(\Omega)}\otimes x^*, U(\chi_{\Omega}\otimes x)\Bigr\rangle \right|\leq v(U)=0.\]
Since $\bigl[T(\chi_\Omega\otimes x)\bigr](t)=R(x)$ on $\Omega$, we have
\begin{align*}
\|R+\mathcal{Z}(X)\|\leq \|R+V\|&=\sup_{x\in S_X}\|Rx+Vx\| \\
&= \sup_{x\in S_X}\left\|Rx +  \frac{1}{\mu(\Omega)}\int_{\Omega} [U(\chi_\Omega \otimes x)](t)  \, d\mu(t)\right\| \\
&= \sup_{x\in S_X}\left\|  \frac{1}{\mu(\Omega)}\int_{\Omega} \bigl([T( \chi_{\Omega}\otimes x)](t)+[U( \chi_{\Omega}\otimes x)](t)\bigr) \, d\mu(t)\right\|\\
&\leq \sup_{x\in S_X} \int_{\Omega} \left\|\left[T\bigl(\frac{\chi_{\Omega}}{\mu(\Omega)} \otimes x\bigr)\right](t) + \left[U\bigl(\frac{\chi_{\Omega}}{\mu(\Omega)} \otimes x\bigr)\right](t) \right\|\, d\mu(t)\\
&= \sup_{x\in S_X} \left\|T\bigl(\frac{\chi_{\Omega}}{\mu(\Omega)} \otimes x\bigr) + U\bigl(\frac{\chi_{\Omega}}{\mu(\Omega)} \otimes x\bigr) \right\|_1\\
&\leq \|T+U\|.
\end{align*}
Therefore we get $\|R+\mathcal{Z}(X)\|\leq \|T+\mathcal{Z}(L_1(\mu,X))\|$ and so $n'(L_1(\mu, X))\leq n'(X)$.
\end{proof}

Our last result in this section shows that there is no non-trivial compact space for which equality in Proposition~\ref{newindexC} always holds.

\begin{proposition}\label{prop:n'ofckell22<1}
Let $K$ be a compact Hausdorff topological space with cardinality greater than one. Then
\[
n'(C(K, \ell_2^2))\leq \frac{\sqrt{3}}{2}<1.
\]
\end{proposition}

We need to establish some preliminary results. Note that $C(K, \ell_2^2)$ can be identified with $C(K)\oplus C(K)$ with the norm
\[
\|(f,g)\|=\sup_{t\in K} \sqrt{f(t)^2+g(t)^2} \qquad \bigl((f,g)\in C(K)\oplus C(K)\bigr).
\]
Given $T\in \mathcal{L}(X)$, we write
\[
\bigl[T(f, g)\bigr](t) = \bigl(\bigl[T_1(f,g)\bigr](t), \bigl[T_2(f,g)\bigr](t)\bigr) \qquad \bigl(t\in K,\ (f,g)\in C(K)\oplus C(K)\bigr)
\]
for convenient bounded linear operators $T_1,T_2:C(K)\oplus C(K) \longrightarrow C(K)$. For $t\in K$ fixed, each $T_i(\cdot, \cdot)(t)$ is a continuous linear functional on $C(K)\oplus C(K)$, and $\bigl[C(K)\oplus C(K)\bigr]^*=C(K)^*\oplus C(K)^*$. So, by the Riesz representation theorem, we have the following representation result:
\[
\bigl[T_i(f,g)\bigr](t) = \left(\int_K f\, d\mu_{i1}^t, \int_K g \, d\mu_{i2}^t\right) \qquad \bigl((f,g)\in C(K)\oplus C(K)\bigr),
\]
where $\mu_{ij}^t$ are Borel regular measures on $K$.

\begin{lemma}\label{basic1}
Let $\mu$ be a Borel regular measure on a compact Hausdorff space $K$ and $0\leq a\leq 1$ be a real number. Suppose that for $t\in K$ we have $\displaystyle\int_K f\, d\mu=0$ for every $f\in S_{C(K)}$ with $f(t) =a$. If $0<a\leq 1$, then $\mu=0$. If $a=0$, then $\mu = c \delta_t$ for some real number $c$.
\end{lemma}

\begin{proof}
Suppose first that $a>0$ and we will show that $\mu=0$. If $\mu$ takes only nonnegative values or takes only nonpositive values, then it is clear. So assume that the positive and negative parts are both nontrivial. Take the Hahn decomposition of $\mu$, $K=\Omega_1\cup \Omega_2$ and assume that the point $t$ is in $\Omega_1$. Then there are compact subsets $K_i$ of $\Omega_i$ such that $|\mu|(\Omega_i\setminus K_i) \leq \beta |\mu|(\Omega_i)$, where $|\mu|(A)$ is the total variation of $\mu$ on the subset $A$ and $\beta = \frac{\sqrt{1+2a}-1}{2}\in ]0,1[$. We may assume that $t\in K_1$. Then  choose two  disjoint open subsets $U_1$ and $U_2$ such that $K_i \subset U_i$ and use Urysohn lemma to choose two continuous functions $f_i:K\longrightarrow [0,1]$ for $i=1,2$, such that
\[ \chi_{K_i} \leq f_i\  \ \ \text{and}\ \  \supp(f_i)\subset  {U_i}\ \ \ (i=1,2).\] Set $f=af_1 - f_2$. Then  $f(t)=a$ and $f\in S_{C(K)}$. However, we get
\begin{align*}
 \int_K f\, d\mu &= \int_{K_1\cup K_2} f\, d\mu +\int_{\Omega_1\setminus K_1} f\, d\mu +\int_{\Omega_2\setminus K_2} f\, d\mu \\
 &\ge a|\mu|(K_1\cup K_2) - |\mu|(\Omega_1\setminus K_1)-|\mu|(\Omega_2\setminus K_2)\\
 &\ge \frac{a}{\beta+1}|\mu|(\Omega_1)+ \frac{a}{\beta+1}|\mu|(\Omega_2)-  \beta|\mu|(\Omega_1)-  \beta|\mu|(\Omega_2)\\
& = \frac{a-\beta(\beta+1)}{\beta+1}|\mu|(K)>0.
\end{align*}
This is a contradiction. When $t$ is in $\Omega_2$, we can show the contradiction similarly.

Suppose now that $a=0$ and, by homogeneity, we have
\[
 \int_{K} f \, d\mu=0
 \]
for every $f\in C(K)$ with $f(t)=0$. This means that the kernel of $\delta_t\in C(K)^*$ is contained in the kernel of $\mu\in C(K)^*$, but this implies $\mu=c\delta_t$ for some real number $c$ as desired.
\end{proof}

The next preliminary result gives us the form of the skew-hermitian operators on $C(K, \ell_2^2)$.

\begin{proposition}\label{Lie2}
Let $X=C(K, \ell_2^2)$ and $T\in \mathcal{Z}(X)$. Then, there is $\lambda\in C(K)$ such that
\[ \bigl[T(f,g)\bigr](t) = (\lambda(t) g(t), - \lambda(t) f(t)).\]
\end{proposition}

\begin{proof}
Fix $t\in K$ and suppose that $f(t)^2+g(t)^2=1$ and $(f, g)\in X$. Then
\[
\bigl\langle T(f,g), \delta_t\otimes (f(t), g(t))\bigr\rangle =0.
\]
That is,
\begin{equation}\label{eq:ckl22}
 f(t)\left(\int_K f\, d\mu_{11}^t+ \int_K g \, d\mu_{12}^t\right) + g(t)\left(\int_K f\, d\mu_{21}^t+ \int_K g \, d\mu_{22}^t\right)=0.
\end{equation}
If we consider $g=0$, for each $f\in S_{C(K)}$ with $f(t)=1$ we have \[
\int_K f\, d\mu_{11}^t =0,
\]
and then $\mu_{11}^t=0$ by Lemma~\ref{basic1}. Anogously, we have $\mu_{22}^t=0$. Now suppose that $f$ in $S_{C(K)}$ satisfies $f(t)=1/\sqrt{2}$ and consider the pair $(f,f)$. Then we have
\[
0= \int_K f\, d\mu_{12}^t + \int_K f \, d\mu_{21}^t= \int_K f\, d\bigl(\mu_{12}^t + \mu_{21}^t\bigr),
 \]
and Lemma~\ref{basic1} gives us that $\mu_{12}^t = -\mu_{21}^t$.

Now, for any function $f\in S_{C(K)}$ with $f(t)=0$, consider $g=\sqrt{1-f^2}\in B_{C(K)}$ which satisfies $g(t)=1$ and $f(t)^2+g(t)^2=1$, and apply \eqref{eq:ckl22} to get that
\[
\int_{K}f\, d\mu_{21}^t =0.
\]
Another application of Lemma~\ref{basic1} gives us that $\mu_{21}^t = \lambda(t) \delta_t$. By using constant functions, it is easy to see that $\lambda\in C(K)$. This completes the proof.
\end{proof}

\begin{proof}[Proof of Proposition~\ref{prop:n'ofckell22<1}]
Suppose that $t_1, t_2$ are two different elements of $K$ and $X=C(K, \ell_2^2)$. Define the operator $T\in \mathcal{L}(X)$ by
\[ T(f, g) = (f, \sqrt{2}f(t_2)).\]
Set
\[
\Gamma = \Bigl\{ \bigl((f, g), \delta_t\otimes (f(t), g(t))\bigr)\in \Pi(X)\,:\, (f, g)\in S_X,\,  f(t)^2 + g(t)^2 =1\Bigr\}.
\]
Then $v(T)$ can be computed using only elements of $\Gamma$ by Lemma~\ref{lemma:old-projection-dense}. Given $\bigl((f, g), \delta_t\otimes (f(t), g(t))\bigr)$ in $\Gamma$, we have
\begin{align*}
\bigl|\bigl\langle \delta_t\otimes (f(t), g(t)), T(f,g)  \bigr\rangle\bigr| &= |f(t)^2 + \sqrt{2}f(t_2)g(t)|\\
&\leq f(t)^2 + \sqrt{2}|g(t)|\leq 1-g(t)^2 +\sqrt{2}|g(t)|\leq 3/2.
\end{align*}
Therefore, $v(T)\leq 3/2$. Finally, we claim that $\|T+\mathcal{Z}(X)\|\ge \sqrt{3}$. Indeed, fix $S\in \mathcal{Z}(X)$ and use Proposition~\ref{Lie2} to see that there is $\lambda \in C(K)$ such that $S(f,g) = (\lambda g, -\lambda f)$ for every $(f,g)\in X$. Now, take a function $f_1\in S_{C(K)}$ such that $f_1(t_1)\lambda(t_1)\leq 0$, $|f_1(t_1)|=1$, $f_1(t_2)=1$, and consider the pair $(f_1, 0)\in S_X$. Then
\begin{align*}
 \|T+S\|^2&\ge \bigl\|\bigl([T+S](f_1,0)\bigr)(t_1)\bigr\|^2\\
 &=f_1(t_1)^2 + \bigl(\sqrt{2}f_1(t_2) - \lambda(t_1)f_1(t_1)\bigr)^2\\
&=1+\bigl(\sqrt{2}-\lambda(t_1)f_1(t_1)\bigr)^2 \geq 3.
\end{align*}
Since $S$ is arbitrary in $\mathcal{Z}(X)$, $\|T+\mathcal{Z}(X)\|\ge \sqrt 3$. This completes the proof.
\end{proof}

\section{Absolute sums of Banach spaces}\label{sec:absolute-sums}
In this section we provide all the results about absolute sums of Banach spaces which we will need later on. We start by recalling the needed notation and we obtain estimations of the numerical index and the second numerical index of absolute sums of Banach spaces.

Let us first introduce the needed notation. Let $\Lambda$ be a nonempty set and let $E$ be a linear subspace
of $\R^\Lambda$. An \emph{absolute norm} on $E$ is a complete norm
$\|\cdot\|_E$ satisfying
\begin{itemize}
\item[(a)] Given $a, b\in \R^\Lambda$ with
    $|a(\lambda)| = |b(\lambda)|$ for every $\lambda\in
    \Lambda$, if $a\in E$, then $b\in E$
    with $\|a\|_E = \|b\|_E$.
\item[(b)] For every $\lambda\in \Lambda$,
    $e_\lambda\in E$ with
    $\|e_\lambda\|_E=1$, where
    $e_\lambda$ is the characteristic function of
    the singleton $\{\lambda\}$.
\end{itemize}
In such a case, we will say that $E$ has \emph{absolute structure}. The following results can be deduced from the definition above:
\begin{itemize}
\item[(c)] Given $a, b\in \R^\Lambda$ with
    $|a(\lambda)| \leq |b(\lambda)|$ for every $\lambda\in
    \Lambda$, if $b\in E$, then $a\in E$
    with $\|a\|_E \leq \|b\|_E$.
\item[(d)] $\ell_1(\Lambda)\subseteq E \subseteq
    \ell_\infty(\Lambda)$ with contractive inclusions.
\end{itemize}
For $a\in E$ we will write $|a|\in E$ to denote the element given by $|a|(\lambda)=|a(\lambda)|$ for every $\lambda\in \Lambda$.

Observe that $E$ is a Banach lattice in the pointwise order
(actually, $E$ can be viewed as a K\"{o}the space on the measure
space $(\Lambda,\mathcal{P}(\Lambda),\nu)$ where $\nu$ is the
counting measure on $\Lambda$, which is non-necessarily
$\sigma$-finite). Thus, we say that $a\in E$ is \emph{positive} if $a(\lambda)\geq 0$ for every $\lambda\in \Lambda$. An operator $U\in \mathcal{L}(E)$ is \emph{positive} if $U(a)$ is positive for every positive $a\in E$. In such a case, it is clear that
$$
\|U\| =\sup \{\|U(a)\| \ : \ a\in E,\, a \text{ positive}\}.
$$
Besides, one has that $\big|U(a)\big|\leq U(|a|)$ for every $a\in E$.

The \emph{K\"{o}the
dual} $E'$ of $E$ is the linear subspace of $\R^\Lambda$
defined by
$$
E'=\left\{b\in \R^\Lambda\ : \ \|b\|_{E'}:=
\sup_{a\in B_E} \sum_{\lambda\in\Lambda}|b(\lambda)||a(\lambda)| <\infty \right\}.
$$
The norm $\|\cdot\|_{E'}$ on $E'$ is an absolute norm. Every
element $b\in E'$ defines naturally a continuous
linear functional on $E$ by the formula
$$
\langle b, a \rangle= \sum_{\lambda\in\Lambda}b(\lambda) a(\lambda) \qquad \bigl(a\in E\bigr),
$$
so we have $E'\subseteq E^*$ and this inclusion is isometric.

Given an arbitrary family $\{X_\lambda\, : \, \lambda\in\Lambda\}$
of Banach spaces, for $x\in \prod_{\lambda\in\Lambda}X_\lambda$ we will use the notation $x=(x(\lambda))_{\lambda\in\Lambda}$. For a linear subspace $E$ of $\R^\Lambda$ with absolute norm,
the \emph{$E$-sum} of the family $\{X_\lambda\, : \, \lambda\in\Lambda\}$ is the space
\begin{align*}
\Bigl[\bigoplus_{\lambda\in \Lambda} X_\lambda\Bigr]_{E} &:=
\Bigl\{x\in \prod_{\lambda\in\Lambda}X_\lambda\,:\, (\|x(\lambda)\|)_{\lambda\in\Lambda}\in E \Bigr\}
\end{align*}
endowed with the complete norm
$\|x\|=\|(\|x(\lambda)\|)_{\lambda\in\Lambda}\|_E$. We will use the name
\emph{absolute sum} when the space $E$ is clear from the context. Natural examples of absolute sums are $c_0$-sums and $\ell_p$-sums for $1\leq p\leq \infty$, i.e.\ given a nonempty set $\Lambda$, we consider $E=c_0(\Lambda)$ or $E=\ell_p(\Lambda)$. More examples
are the absolute sums produced using a Banach space $E$ with a
one-unconditional basis, finite (i.e.\ $E$ is $\R^m$ endowed with
an absolute norm) or infinite (i.e.\ $E$ is a Banach space with an
one-unconditional basis viewed as a linear subspace of $\R^\N$ via
the basis).

Let $\Lambda$ be a non-empty set and let $E\subset \R^\Lambda$ be a Banach space with absolute norm. Let $\{X_\lambda\, : \,
\lambda\in\Lambda\}$ be an arbitrary family of Banach spaces and write $X=\Bigl[\bigoplus_{\lambda\in \Lambda} X_\lambda\Bigr]_{E}$.
For every $\kappa\in\Lambda$, we consider
the natural inclusion $I_\kappa:X_\kappa\longrightarrow X$ given
by $I_\kappa(x)=x\,\chi_{\{\kappa\}}$ for every $x\in X_\kappa$.
which is an isometric embedding, and  the image $I_\kappa(x)$ is denoted by $x\otimes e(\kappa)$. The natural projection
$P_\kappa:X \longrightarrow X_\kappa$ is given by
$P_\kappa(x)=x(\kappa)$ for every $x\in X$, which is contractive. Clearly, $P_\kappa I_\kappa=\Id_{X_\kappa}$.
We write
$$
X'=\Bigl[\bigoplus_{\lambda\in \Lambda} X^*_\lambda\Bigr]_{E'}
$$
and we observe that every element $x^*=(x^*(\lambda))_{\lambda\in\Lambda}\in X'$ defines
naturally a continuous linear functional on $X$ by the formula
$$
x\longmapsto \sum_{\lambda\in\Lambda}\langle x^*(\lambda) , x(\lambda)\rangle  \qquad \bigl(x\in X\bigr),
$$
so we have $X'\subseteq X^*$ and this inclusion is isometric as we see in the following.

\begin{proposition}\label{prop:E'=E*=>X'=X*}
Let $Z$ be the space consisting of all elements
$(x^*(\lambda))_{\lambda\in \Lambda}$ in $\prod_{\lambda\in \Lambda} X^*_\lambda$ such that
 $\sum_{\lambda\in\Lambda} |\langle x^*(\lambda) , x(\lambda)\rangle |$ converges for all $(x(\lambda))_{\lambda\in \Lambda}$ in $X$.
Then the following hold:
\begin{enumerate}
\item[(a)]  $Z$ is a Banach space equipped with the norm
\[ \tri (x^*(\lambda))_{\lambda\in\Lambda} \tri = \sup\left\{\Bigl|\sum_{\lambda\in\Lambda} \langle x^*(\lambda) , x(\lambda)\rangle\Bigr|:\, (x(\lambda))_{\lambda\in\Lambda} \in B_X\right\} \]

\item[(b)] $Z$ is isometric to $X'$.
\item[(c)] $Z$ is isometrically embedded in $X^*$: each $(x^*(\lambda))_{\lambda\in \Lambda}\in Z$ defines a continuous linear functional on $X$ by
$$
(x(\lambda))_{\lambda\in\Lambda} \longmapsto \sum_{\lambda\in\Lambda}\langle x^*(\lambda) , x(\lambda)\rangle  \qquad \bigl((x(\lambda))_{\lambda\in\Lambda} \in X\bigr).$$
\end{enumerate}
 Moreover, if $E'=E^*$, then $X'=X^*$.
\end{proposition}

\begin{proof}
Given an element $(x^*(\lambda))_{\lambda\in \Lambda}$ in $Z$, consider the operator $T: X \longrightarrow \ell_1(\Lambda)$ defined by
$$
T\bigl((x(\lambda))_{\lambda\in \Lambda}\bigr) = \bigl(\langle x^*(\lambda) , x(\lambda)\rangle\bigr)_{\lambda\in \Lambda} \qquad \bigl((x(\lambda))_{\lambda\in\Lambda} \in X\bigr).
$$
Using the closed graph theorem, it is easy to see that $T$ is bounded and the linear functional
$$
\varphi\bigl((x(\lambda))_{\lambda\in\Lambda}\bigr) =\sum_{\lambda\in\Lambda}\langle x^*(\lambda) , x(\lambda)\rangle \qquad \bigl((x(\lambda))_{\lambda\in\Lambda} \in X\bigr)
$$
is bounded on $X$. Hence it is clear that (c) holds and $\tri\cdot \tri$ is a well-defined norm on $Z$. We will show that (b) holds and this implies in particular that $Z$ is a Banach space, which completes the proof of (a). Indeed, (b) is shown by the following direct computation:
\begin{align*}
\tri (x^*(\lambda))_{\lambda\in\Lambda} \tri &= \sup\left\{\left|\sum_{\lambda\in\Lambda} \eps(\lambda)\langle x^*(\lambda) , x(\lambda)\rangle\right|:\, (x(\lambda))_{\lambda\in\Lambda} \in B_X,\, \eps(\lambda) =\pm 1\, \forall \lambda \in \Lambda\right\} \\
&=\sup\left\{\sum_{\lambda\in\Lambda} |\langle x^*(\lambda) , x(\lambda)\rangle|:\, (x(\lambda))_{\lambda\in\Lambda} \in B_X \right\}\\
&=\sup\left\{\sum_{\lambda\in\Lambda} |\langle x^*(\lambda) , \|x(\lambda)\|a(\lambda)\rangle|:\, (\|x(\lambda)\|)_{\lambda\in\Lambda} \in B_E,\, a(\lambda)\in S_{X_\lambda}  \,\forall \lambda \in \Lambda\right\}\\
&=\sup\left\{\sum_{\lambda\in\Lambda} \|x(\lambda)\| |\langle x^*(\lambda) ,a(\lambda)\rangle|:\, (\|x(\lambda)\|)_{\lambda\in\Lambda} \in B_E,\, a(\lambda)\in S_{X_\lambda} \, \forall \lambda \in \Lambda\right\}\\
&=\sup\left\{\sum_{\lambda\in\Lambda} \|x(\lambda)\| | \|x^*(\lambda)\|:\, (\|x(\lambda)\|)_{\lambda\in\Lambda} \in B_E \right\}\\
&=\bigl\|(\|x^*(\lambda)\|)_{\lambda\in \Lambda}\bigr\|_{E'} = \bigl\|(x^*(\lambda))_{\lambda\in \Lambda}\bigr\|_{X'}.
\end{align*}
To prove the `moreover' part, note that $E'=E^*$ if and only if $E$ is order continuous. We first show that every element of $X$ is approximated by the finite sum of elements $x(\lambda)\otimes e(\lambda)$'s. Let  $(x(\lambda))_{\lambda\in \Lambda}$ be an element of $X$. Then,  $(\|x(\lambda)\|)_{\lambda\in \Lambda}$ is an element of $E$ and so
\begin{align*} 0&=\lim_{F\in  \mathcal{F}} \left\| \sum_{\lambda \in F} \|x(\lambda)\| \otimes e(\lambda) - (\|x(\lambda) \|)_{\lambda \in \Lambda}\right\|
=\lim_{F\in  \mathcal{F}}
 \left\| \sum_{\lambda \in F} x(\lambda) \otimes e(\lambda) - (x(\lambda) )_{\lambda \in \Lambda}\right\|,
\end{align*} where $\mathcal{F}$ is the family of finite subsets of $\Lambda$ ordered by inclusion. This means that the family $S$ of subsets of $X$ defined by
\[ S = \left\{ \sum_{\lambda \in F} x(\lambda) \otimes e(\lambda) : F\in \mathcal{F}, \Lambda, x(\lambda) \in X_\lambda\  \forall \lambda \in F\right\},\] is dense in $X$.
Now, we claim that if $\varphi\in X^*$, then there is $(x^*(\lambda))_{\lambda \in \Lambda}$ in $X'$ such that
\[ \varphi( (x(\lambda))_{\lambda\in \Lambda} ) =  \sum_{\lambda \in \Lambda} \langle x^*(\lambda), x(\lambda) \rangle.\]
Indeed, for each $\lambda$ in $\Lambda$, define $[x^*(\lambda)](x(\lambda)) = \varphi(x(\lambda) \otimes e(\lambda))$ for all $x(\lambda)\in X_\lambda$ and observe that $x^*(\lambda)\in X_\lambda^*$. For every $(x(\lambda))_{\lambda\in \Lambda}$ in $X$ and for every finite subset $F$ of $\Lambda$, we have
\[ \sum_{\lambda\in F} | \langle x^*(\lambda), x(\lambda) \rangle| = \left| \varphi\left(\sum_{\lambda \in F} \eps(\lambda) x(\lambda) \otimes e(\lambda)\right) \right| \le \|\varphi\|\left\| \sum_{\lambda \in F}  \eps(\lambda) x(\lambda) \otimes e(\lambda)\right\|\le \|\varphi\|\|(\|x(\lambda)\|)_{\lambda\in \Lambda} \|_E, \] where $\eps(\lambda)$'s are suitable scalars of modulus one. So $\sum_{\lambda\in \Lambda} | \langle x^*(\lambda), x(\lambda) \rangle|<\infty$ for all $(x(\lambda))_{\lambda\in \Lambda}$ in $X$ and so $(x^*(\lambda))_{\lambda \in \Lambda}$ is an element of $X'$. Since the linear functional defined by $(x^*(\lambda))_{\lambda \in \Lambda}$ is then equal to $\varphi$ on the dense subset $S$, they are equal on the whole space $X$, proving the claim. Therefore, $X'=X^*$, and this completes the proof.
\end{proof}

We say that an operator $S\in \mathcal{L}(X)$ is \emph{diagonal} if $P_\lambda S I_\mu=0$ whenever $\lambda\neq \mu$. Equivalently, there is a family $S_\lambda\in \mathcal{L}(X_\lambda)$ for each $\lambda\in \Lambda$ such that $S(x)=\big(S_\lambda (x(\lambda))\big)_{\lambda\in \Lambda}$ for every $x\in X$. Observe that if $S\in \mathcal{Z}(X)$ is diagonal, then every element $S_\lambda=P_\lambda S I_\lambda$ in the diagonal belongs to $\mathcal{Z}(X_\lambda)$ since, otherwise, the numerical radius of $S$ is not zero.

The following result follows from Lemma~\ref{lemma:old-Anorming-Ardalani}.

\begin{remark}\label{rem:dense-subset-numericalradius}
Let $E\subset \R^\Lambda$ be a Banach space having an absolute norm such that $E'$ is norming.
Let $\{X_\lambda\, : \,
\lambda\in\Lambda\}$ be an arbitrary  family of Banach spaces and $X=\Bigl[\bigoplus_{\lambda\in \Lambda} X_\lambda\Bigr]_{E}$.
Then, for every $T\in \mathcal{L}(X)$
$$
v(T)=\inf_{\delta>0}\sup\left\{\left|\sum_{\lambda\in \Lambda} \langle x^*(\lambda), [T(x)](\lambda) \rangle \right|\,:\ x\in S_X,\, x^*\in S_{X'},\, \sum_{\lambda\in \Lambda} \langle x^*(\lambda), x(\lambda)\rangle>1-\delta\right\}.
$$
If $T\in \mathcal{L}(E)$ is positive, then
$$
v(T)=\inf_{\delta>0}\sup\left\{\sum_{\lambda\in \Lambda} b(\lambda)[T(a)](\lambda)\,:\ a\in S_E,\, b\in S_{E'},\, \text{$a$, $b$ positive},\, \sum_{\lambda\in \Lambda} b(\lambda)a(\lambda)>1-\delta\right\}.
$$
\end{remark}

The next result tells us that the second numerical index of an absolute sum of Banach spaces is bounded above by the infimum of the numerical indices of the addends, provided that $E'$ is norming and that all skew-hermitian operators are diagonal. It is the natural analogue of \cite[Theorem~2.1]{MarMerPopRan} for the second numerical index.

\begin{prop}\label{prop:derived-index-summands}
Let $\Lambda$ be a non-empty set and let $E\subset \R^\Lambda$ be a Banach space with absolute norm such that $E'$ is norming.
Let $\{X_\lambda\, : \, \lambda\in\Lambda\}$ be an arbitrary  family of Banach spaces and $X=\Bigl[\bigoplus_{\lambda\in \Lambda} X_\lambda\Bigr]_{E}$. Suppose that every $S\in \mathcal{Z}(X)$ is diagonal. Then,
$$
n'\left(X\right)\leq \inf\bigl\{n'(X_\lambda)\,:\,\lambda\in \Lambda\bigr\}.
$$
\end{prop}

\begin{proof}
For a fixed $\kappa\in \Lambda$, we will show that $n'\left(X\right)\leq n'(X_\kappa)$. To this end, given
$U\in \mathcal{L}(X_\kappa)$, we define $T\in \mathcal{L}(X)$ by $T=I_{\kappa} U
P_\kappa$. In the proof of \cite[Theorem~2.1]{MarMerPopRan} it is shown that $v(T)\leq v(U)$ using a stronger hypothesis on $E'$ to be able to apply Lemma~\ref{lemma:old-projection-dense}. Using the same proof, replacing
Lemma~\ref{lemma:old-projection-dense} by Lemma~\ref{lemma:old-Anorming-Ardalani} (actually its consequence Remark~\ref{rem:dense-subset-numericalradius}), one also obtains $v(T)\leq v(U)$ under our hypotheses.

All that remains to be proved is $\|T+\mathcal{Z}(X)\|\geq \|U+\mathcal{Z}(X_\kappa)\|$. To do so, fixed $S\in \mathcal{Z}(X)$, observe that by hypothesis we have that $P_\kappa S I_\kappa=S_\kappa$ and $S_\kappa \in \mathcal{Z}(X_\kappa)$. Moreover, since $P_\kappa T I_\kappa=U$ we can write
$$
\|U+\mathcal{Z}(X_\kappa)\|\leq \|U+S_\kappa\|=\|P_\kappa(T+S)I_\kappa\|\leq \|T+S\|
$$
which gives $\|U+\mathcal{Z}(X_\kappa)\|\leq \|T+\mathcal{Z}(X)\|$ just taking infimum on $S\in \mathcal{Z}(X)$.
\end{proof}

The following result gives an inequality between the numerical
index of an $E$-sum of Banach spaces and the normalized numerical radius of positive operators on the space $E$, provided that $E'$ contains sufficiently many functionals.

\begin{prop}\label{prop:EMA-infinito}
Let $\Lambda$ be a non-empty set and let $E\subset \R^\Lambda$ be a Banach space with absolute norm such that $E'$ is norming.
Let $\{X_\lambda\, : \,
\lambda\in\Lambda\}$ be an arbitrary  family of Banach spaces and $X=\Bigl[\bigoplus_{\lambda\in \Lambda} X_\lambda\Bigr]_{E}$. Then, for every positive operator $U\in \mathcal{L}(E)$,
$$
n(X) \leq \frac{v(U)}{\|U\|}\,.
$$
\end{prop}

\begin{proof}
For each $\lambda\in \Lambda$ we fix $(y_\lambda,y^*_\lambda)\in \Pi(X_\lambda)$ and for each $x\in X$ we consider the element $a_x \in E$ defined by $a_x(\lambda)=y^*_\lambda(x(\lambda))$ for $\lambda\in \Lambda$. Now, fixed a positive operator $U\in \mathcal{L}(E)$, we define $T\in \mathcal{L}(X)$ by
$$
\big[T(x)\big](\lambda)=\big[U(a_x)\big](\lambda) y_\lambda \qquad \big(x\in X,\lambda\in \Lambda\big).
$$
We start checking that $\|T\|=\|U\|$. Indeed, fixed $x\in S_X$, we can write
\begin{align*}
\|T(x)\|&=\left\|\left(\big\|\big[U(a_x)\big](\lambda) y_\lambda\big\|\right)_{\lambda\in \Lambda}\right\|_E=\left\|\left(\big|\big[U(a_x)\big](\lambda)\big|\right)_{\lambda\in \Lambda}\right\|_E=\big\||U(a_x)|\big\|\\
&\leq  \|U\| \|a_x\| \leq \|U\|\|x\|=\|U\|
\end{align*}
which gives $\|T\|\leq \|U\|$. To prove the reversed inequality, fixed $a\in S_E$ positive, we define $x\in S_X$ by $x(\lambda)=a(\lambda) y_\lambda$ for every $\lambda\in\Lambda$, and we observe that $a_x=a$. Therefore, using that $U$ is positive, we have that
$$
\|T\|\geq \|T(x)\|=\left\|\left(\big\|\big[U(a_x)\big](\lambda) y_\lambda\big\|\right)_{\lambda\in \Lambda}\right\|_E=\left\|\left(\big|\big[U(a)\big](\lambda)\big|\right)_{\lambda\in \Lambda}\right\|_E=\big\||U(a)|\big\|=\|U(a)\|
$$
which gives $\|T\|\geq \|U\|$ by taking supremum on $a$. Now we turn to prove that $v(T)\leq v(U)$. By Remark~\ref{rem:dense-subset-numericalradius}, fixed $\eps>0$, there is $\delta>0$
such that
$$
\sup\left\{\sum_{\lambda\in \Lambda} b(\lambda)[U(a)](\lambda)\,:\ a\in S_E,\, b\in S_{E'},\, \text{$a$, $b$ positive},\, \sum_{\lambda\in \Lambda} b(\lambda)a(\lambda)>1-\delta\right\}\leq v(U)+\eps.
$$
Fixed $x\in S_X$ and $x^*\in S_{X'}$ satisfying $\sum_{\lambda\in \Lambda} \langle x^*(\lambda) ,x(\lambda)\rangle>1-\delta$ we consider $a=(||x(\lambda)||)_{\lambda\in\Lambda}\in S_E$ and $b=(||x^*(\lambda)||)_{\lambda\in\Lambda}\in S_{E'}$ which satisfy $b(a)>1-\delta$. Since $|a_x|\leq a$ and $U$ is positive, we obtain
$$
\big|U(a_x)\big|\leq U\big(|a_x|\big) \leq U(a).
$$
So we can estimate as follows
\begin{align*}
|x^*(Tx)| &=\left|\sum_{\lambda\in \Lambda}\big[U(a_x)\big](\lambda) \langle x^*(\lambda), y_\lambda \rangle\right|\leq \sum_{\lambda\in \Lambda}\left|\big[U(a_x)\big](\lambda)\right| \left|\langle x^*(\lambda), y_\lambda \rangle\right|\\
&\leq \sum_{\lambda\in \Lambda}\|x^*(\lambda)\|\left|\big[U(a_x)\big](\lambda)\right|\leq \sum_{\lambda\in \Lambda}\|x^*(\lambda)\|\big[U(|a_x|)\big](\lambda) \\
&\leq \sum_{\lambda\in \Lambda}\|x^*(\lambda)\|\big[U(a)\big](\lambda) =\langle b, U(a)\rangle\leq v(U)+\eps.
\end{align*}
Therefore, using Remark~\ref{rem:dense-subset-numericalradius} we can write
$$
v(T)\leq \sup\left\{\left|\sum_{\lambda\in \Lambda} \langle x^*(\lambda), [T(x)](\lambda) \rangle \right|\,:\ x\in S_X,\, x^*\in S_{X'},\, \sum_{\lambda\in \Lambda} \langle x^*(\lambda),x(\lambda)\rangle>1-\delta\right\}\leq v(U)+\eps
$$
and the arbitrariness of $\eps$ gives $v(T)\leq v(U)$. Finally, we observe that
$$
v(U)\geq v(T)\geq n(X)\|T\|\geq n(X)\|U\|,
$$
which finishes the proof.
\end{proof}

The next result gives an analogue of Proposition~\ref{prop:EMA-infinito} for the second numerical index provided that we restrict a bit more the class of operators considered on $E$.

\begin{proposition}\label{prop:derived-index-positive-operators}
Let $\Lambda$ be a non-empty set and let $E\subset \R^\Lambda$ be a Banach space with absolute norm so that $E'$ is norming.
Let $\{X_\lambda\, : \,
\lambda\in\Lambda\}$ be an arbitrary  family of Banach spaces and $X=\Bigl[\bigoplus_{\lambda\in \Lambda} X_\lambda\Bigr]_{E}$. Suppose that every $S\in \mathcal{Z}(X)$ is diagonal. Let $U\in \mathcal{L}(E)$ be a positive operator and suppose that there is a sequence $\{a_n\}$ in $E$ such that $\lim \|Ua_n\|=\|U\|$ and $\supp (Ua_n)\cap \supp(a_n)=\emptyset$ for every $n\in \N$.
Then,
$$
n'(X)\leq \frac{v(U)}{\|U\|}\,.
$$
\end{proposition}

\begin{proof}
For each $\lambda\in \Lambda$ we fix $(y_\lambda,y^*_\lambda)\in \Pi(X_\lambda)$ and for each $x\in X$ we consider the element $a_x \in E$ defined by $a_x(\lambda)=y^*_\lambda(x(\lambda))$ for $\lambda\in \Lambda$. We define $T\in \mathcal{L}(X)$ by
$$
\big[T(x)\big](\lambda)=\big[U(a_x)\big](\lambda) y_\lambda \qquad \big(x\in X,\lambda\in \Lambda\big).
$$
It is shown in Proposition~\ref{prop:EMA-infinito} that $v(T)\leq v(U)$ so it is enough to prove that $\|T +\mathcal{Z}(X)\|\geq \|U\|$. Fixed $S\in \mathcal{Z}(X)$, for every $\lambda\in \Lambda$ there is $S_\lambda\in \mathcal{Z}(X_\lambda)$  such that $S(x)=\big(S_\lambda (x(\lambda))\big)_{\lambda\in \Lambda}$ for every $x\in X$. For each $n\in \N$ we consider $x_n\in X$ given by $x_n(\lambda)=a_n(\lambda)y_\lambda$ for $\lambda\in\Lambda$ which satisfies $a_{x_n}=a_n$. Now we can write
\begin{align*}
\|T+S\|&\geq \|[T+S](x_n)\|=\left\|\Big(\big[U(a_{x_n})\big](\lambda) y_\lambda\Big)_{\lambda\in \Lambda}+\Big(S_\lambda (x_n(\lambda))\Big)_{\lambda\in \Lambda}\right\|\\
&=\left\|\Big(\big[U(a_{x_n})\big](\lambda) y_\lambda\Big)_{\lambda\in \Lambda}+\Big(a_n(\lambda)S_\lambda (y_\lambda)\Big)_{\lambda\in \Lambda}\right\|\\
&\geq \left\|\Big(\big[U(a_{x_n})\big](\lambda) y_\lambda\Big)_{\lambda\in \Lambda}\right\|=\left\|U(a_{x_n})\right\|_E=\left\|U(a_n)\right\|_E
\end{align*}
where we used in the last inequality that  $\supp (Ua_n)\cap \supp(a_n)=\emptyset$ and so
$$
\supp\Big(\big[U(a_{x_n})\big](\lambda) y_\lambda\Big)_{\lambda\in \Lambda}\bigcap\supp\Big(a_n(\lambda)S_\lambda (y_\lambda)\Big)_{\lambda\in \Lambda}=\emptyset.
$$
Taking limit in the above inequality, we obtain $\|T+S\|\geq \|U\|$ and taking infimum on $S\in \mathcal{Z}(X)$ we get $\|T+\mathcal{Z}(X)\|\geq \|U\|$ as desired.
\end{proof}

Observe that when $E$ has dimension two, the above result only applies to the multiples of the two shift operators. However this is enough to obtain an obstructive result about the absolute norm in $E$ when $n'(X)=1$. We gather the remaining work needed to achieve that result in the following lemma which may have independent interest.

\begin{lemma}\label{lem-shift-operators}
Let $E=(\R^2,|\cdot|)$ where $|\cdot|$ is an absolute norm, let $\{e_1, e_2\}$ be the standard unit basis in $E$, let $(e_1^*, e_2^*)$ be its dual basis, and let $U_1=e_2^*\otimes e_1$ and $U_2=e_1^*\otimes e_2$ be the shift operators. Suppose that there is $0<k\leq1$ such that $\min\{v(U_1),v(U_2)\}\geq k$. Then,
$$
\max\left\{|e_1+e_2|,|e_1^*+e_2^*|\right\}\geq 1-\sqrt{1-k^2}+k.
$$
Consequently, if $k=1$ it follows that either $|\cdot|=\|\cdot\|_1$ or $|\cdot|=\|\cdot\|_\infty$.
\end{lemma}

\begin{proof}
Since $v(U_1)\geq k$ and $U_1$ is a positive operator, there are $a,b,\alpha,\beta\in [0,1]$ such that the elements $x_1=ae_1+be_2\in S_{E}$ and $x_1^*=\alpha e_1^*+\beta e_2^*\in S_{E^*}$ satisfy $x_1^*(x_1)=1$ and
$$
k\leq x_1^*(U_1x_1)=x_1^*(e_1)e_2^*(x_1)=\alpha b
$$
so we obtain $\alpha\geq k$ and $b\geq k$. An analogous argument for $U_2$ gives us the existence of $c,d,\lambda,\mu\in[0,1]$ such that $x_2=ce_1+de_2\in S_E$ and $x_2^*=\lambda e_1^*+\mu e_2^*\in S_{E^*}$ satisfy $x_2^*(x_2)=1$ and
$$
k\leq x_2^*(U_2x_2)=x_2^*(e_2)e_1^*(x_2)=\mu c
$$
so we also obtain $\mu\geq k$ and $c\geq k$. Now observe that
\begin{align*}
1&\geq x_1^*(x_2)=\alpha c+\beta d\geq k^2+\beta d
\end{align*}
which implies $\beta d\leq 1-k^2$. Therefore, we get that $\min\{\beta,d\}\leq\sqrt{1-k^2}$. If $d\leq \sqrt{1-k^2}$ then
$$
1=x_2^*(x_2)=\lambda c+\mu d\leq \lambda+d\leq \lambda +\sqrt{1-k^2}
$$
and we get that $\lambda\geq 1-\sqrt{1-k^2}$. Therefore, we can write
$$
1-\sqrt{1-k^2}+k\leq\lambda+\mu=x_2^*(e_1+e_2)\leq |x_2^*||e_1+e_2|=|e_1+e_2|
$$
which finishes the proof in this case. Analogously, if $\beta\leq \sqrt{1-k^2}$ then
$$
1=x_1^*(x_1)=\alpha a+\beta b\leq a+\beta\leq a+ \sqrt{1-k^2}
$$
and we obtain $a\geq 1-\sqrt{1-k^2}$. Hence,
$$
1-\sqrt{1-k^2}+k\leq a+b=[e^*_1+e^*_2](x_1)\leq |e^*_1+e^*_2||x_1|=|e^*_1+e^*_2|
$$
and the proof is complete.
\end{proof}

\begin{corollary}\label{cor:n'-1-sum-2-subspaces}
Let $X$ be a Banach space and $Y$, $W$ non-trivial closed subspaces of $X$ such that $X=Y \oplus_a W$, where $\oplus_a$ is an absolute sum. If $n'(X)=1$, then $\oplus_a=\oplus_\infty$ or $\oplus_a=\oplus_1$ or $\oplus_a=\oplus_2$.
\end{corollary}

\begin{proof}
Let $U_1$ and $U_2$ be the two normalized shift operators on $E=(\R^2,|\cdot|_a)$ as in Lemma~\ref{lem-shift-operators}. If $\oplus_a\neq \oplus_2$, Lemma~\ref{lemma:lemma5Paya82} tells us that every element in $\mathcal{Z}(X)$ is diagonal and, therefore, Proposition~\ref{prop:derived-index-positive-operators} can be applied. It gives that
$$
1=n'(X)\leq \min\left\{\frac{v(U_1)}{\|U_1\|}\,,\frac{v(U_2)}{\|U_2\|}\right\}.
$$
Finally, Lemma~\ref{lem-shift-operators} finishes the proof.
\end{proof}
We finish the section showing that for three-dimensional spaces with numerical index zero, the set of values of the second numerical index is not an interval. To prove this result we need the following easy lemma.

\begin{lemma}\label{lem:R2-close-to-ell1}
Let $E=(\R^2,|\cdot|)$ where $|\cdot|$ is an absolute norm and let $\{e_1, e_2\}$ be the standard unit basis in $E$. If $|e_1+e_2|=\xi$ for $1\leq\xi\leq 2$, then
$$
|ae_1+be_2|\leq \|ae_1+be_2\|_1\leq (3-\xi)|ae_1+be_2| \qquad \forall a,b \in \R.
$$
\end{lemma}

\begin{proof} Observe that the first inequality always holds for absolute norms. To prove the second one, using homogeneity and that $|\cdot|$ is an absolute norm, it is enough to show that $\|ae_1+be_2\|_1\leq 3-\xi$ for $ae_1+be_2\in B_E$ with $a,b\in[0,1]$. So let $ae_1+be_2\in B_E$ with $a,b\in[0,1]$ be fixed. We claim that
$$
\text{if} \quad a\leq b \quad \text{then} \quad b(\xi-1)\leq 1-a \qquad \text{and}\qquad \text{if} \quad b\leq a \quad \text{then} \quad a(\xi-1)\leq 1-b.
$$
Indeed, suppose first that $a\leq b$, call $\alpha=\frac{1}{1+b-a}\in [0,1]$, and observe that
$$
\frac{b}{1+b-a}(e_1+e_2)=\alpha(ae_1+be_2)+(1-\alpha)e_1 \in B_E.
$$
This, together with the fact that $|e_1+e_2|=\xi$, tells us that $\frac{b\xi}{1+b-a}\leq 1$ and, therefore, $b(\xi-1)\leq 1-a$. The case $b\leq a$ is proved analogously using this time $\alpha=\frac{1}{1+a-b}$ and the element $\alpha(ae_1+be_2)+(1-\alpha)e_2 \in B_E$.

To finish the proof, suppose that $a\leq b$ and use the claim to write
\begin{align*}
\|ae_1+be_2\|_1=a+b=a+b(\xi-1)+b(2-\xi)\leq a+1-a+2-\xi=3-\xi.
\end{align*}
If $b\leq a$ a similar argument gives the result.
\end{proof}

\begin{prop}\label{prop:Obstructive-dim-three}
Let $\mathcal{A}=\{X \text{ Banach space} \, : \, \dim(X)=3,\, n(X)=0\}$. Then,
$$
\sup\{n'(X) \, : \, X\in\mathcal{A},\, X\neq\ell_2^3 \}<1.
$$
Therefore, $\{n'(X) \, : \, X\in\mathcal{A}\}\neq (0,1]$ and so, it is not an interval.
\end{prop}

\begin{proof}
Suppose to the contrary that there is a sequence $\{X_j\}\subset\mathcal{A}$ such that $\lim_{j\to \infty} n'(X_j)=1$ and $X_j\neq \ell_2^3$ for each $j\in \N$. Since $\dim(X_j)=3$ and $n(X_j)=0$ for every $j\in \N$, we can use \cite[Corollary~2.5]{MaMeRo} for each $j\in \N$ to get the existence of an absolute norm $|\cdot|_{a_j}$ on $\R^2$ such that $X_j=\ell_2^2\oplus_{a_j} \R$. Moreover, since $X_j\neq\ell_2^3$, we have that $\oplus_{a_j}\neq\oplus_2$ so we can apply Lemma~\ref{lemma:lemma5Paya82} to get that every element in $\mathcal{Z}(X_j)$ is diagonal. Thus Proposition~\ref{prop:derived-index-positive-operators} can be applied. It gives, taking into account that $U_1$ and $U_2$ in Lemma~\ref{lem-shift-operators}  have norm one for every $|\cdot|_{a_j}$, that
$$
n'(X_j)\leq \min\left\{v(U_1),v(U_2)\right\}.
$$
Now we can use Lemma~\ref{lem-shift-operators} for each $j\in \N$ to get
$$
\max\left\{|e_1+e_2|_{a_j},|e_1^*+e_2^*|_{a_j}\right\}\geq 1-\sqrt{1-\big(n'(X_j)\big)^2}+n'(X_j)
$$
and, therefore,
$$
\lim_{j\to\infty}\max\left\{|e_1+e_2|_{a_j},|e_1^*+e_2^*|_{a_j}\right\}=2.
$$
Passing to a convenient subsequence we may and do assume that either $\lim_{j\to\infty}|e_1+e_2|_{a_j}=2$ or  $\lim_{j\to\infty}|e_1^*+e_2^*|_{a_j}=2$.
This, together with Lemma~\ref{lem:R2-close-to-ell1}, tells us that either $\{(\R^2,|\cdot|_{a_j})\}$ or  $\{(\R^2,|\cdot|_{a_j})^*\}$ converges to $\ell_1^2$ in the Banach-Mazur distance. So $\{X_j\}$ converges to $\ell_2^2\oplus_1\R$ or to $\ell_2^2\oplus_\infty\R$ in the Banach-Mazur distance. Now observe that, since $\oplus_{a_j}\neq\oplus_2$, all the spaces $X_j's$ and the spaces  $\ell_2^2\oplus_1\R$ and $\ell_2^2\oplus_\infty\R$ share the same Lie algebra:
$$
\left\{\begin{pmatrix}0&\lambda & 0\\ -\lambda & 0 & 0\\ 0& 0 & 0
\end{pmatrix}\ : \ \lambda\in \R\right\}.
$$
Hence, we can use Proposition~\ref{prop:continuity-B-M-distance-same-Lie-algebra} and Example~\ref{example:Hilbert-sum-infty-Hilbert} to get that
$$
\lim_{j\to\infty }n'(X_j)=n'(\ell_2^2\oplus_1\R)\leq \frac{\sqrt{3}}{2} \qquad \text{or} \qquad \lim_{j\to\infty} n'(X_j)=n'(\ell_2^2\oplus_\infty\R)\leq \frac{\sqrt{3}}{2}
$$
which gives the desired contradiction.
\end{proof}

\section{Spaces with one-unconditional basis and second numerical index one}\label{sect-numindex-one}
We devote this section to characterize Banach spaces with one-unconditional basis and second numerical index one. We start recalling the concept of
(long) one-unconditional basis. Let $\Gamma$ be a non-empty set and let $E$ be a Banach space. We say that $\{e_\gamma\}_{\gamma\in \Gamma}$ is a \emph{(long) one-unconditional basis} for $E$ if the following conditions hold:
\begin{itemize}
\item For every $x\in E$ there is a unique family of real numbers $\{a_\gamma\}_{\gamma\in \Gamma}$ such that $x=\sum a_\gamma e_\gamma$ in the sense that for every $\eps>0$ there is a finite set $F\subset \Gamma$ such that
$$
\Big\|x-\sum_{\gamma\in F'}a_\gamma e_\gamma\Big\|\leq\eps
$$
holds for every finite subset $F'$ of $\Gamma$ with $F\subset F'$. Observe that, for fixed $x\in E$, only countably many coordinates $a_\gamma$ are non-zero.
\item Whenever $x=\sum a_\gamma e_\gamma \in E$ and $\{b_\gamma\}$ is a family of scalars satisfying $|b_\gamma|\leq|a_\gamma|$ for every $\gamma\in\Gamma$ it follows that $y=\sum b_\gamma e_\gamma \in E$ and $\|y\|\leq \|x\|$.
\item $\|e_\gamma\|=1$ for every $\gamma\in\Gamma$.
\end{itemize}
Remark that when $\Gamma$ is finite or countable this leads to the usual concept of one-unconditional Schauder basis.

Observe that if $E$ has a one-unconditional basis, it can be seen (via the basis) as a subspace of $\R^\Gamma$ with absolute norm satisfying $E'=E^*$. Therefore, if $\{X_\gamma \, : \, \gamma\in \Gamma\}$ is a family of Banach spaces it makes sense to consider the $E$-absolute sum of the family $X=\Bigl[\bigoplus_{\gamma\in \Gamma} X_\gamma\Bigr]_{E}$ which satisfies $X'=X^*$ (see Proposition~\ref{prop:E'=E*=>X'=X*}).

\begin{theorem}\label{thm:n'-one-implies-Hilbert}
Let $X$ be a Banach space with one-unconditional basis. Suppose that $\mathcal{Z}(X)\neq \{0\}$ and $n'(X)=1$. Then, $X$ is a Hilbert space.
\end{theorem}

Prior to give the proof of this theorem, we state some consequences.

\begin{corollary}\label{cor:n'-equal-1}
Let $X$ be a Banach space with one-unconditional basis. If $n'(X)=1$ then either $n(X)=1$ or $X$ is a Hilbert space.
\end{corollary}

\begin{proof}
If $n(X)<1=n'(X)$, then $\mathcal{Z}(X)\neq \{0\}$. Therefore, we can use Theorem~\ref{thm:n'-one-implies-Hilbert} to get that $X$ is a Hilbert space.
\end{proof}

\begin{corollary}\label{cor:n'-fin-dim}
Let $X$ be a real finite-dimensional space with one-unconditional basis. If $n'(X)=1$ then either $n(X)=1$ or $X$ is a Hilbert space.
\end{corollary}

\begin{corollary}\label{cor:n'-1-not-ell-1}
Let $X$ be an infinite-dimensional real Banach space with one-unconditional basis and such that $n'(X)=1$. Then, either $X^*\supseteq \ell_1$ or $X$ is a Hilbert space.
\end{corollary}

\begin{proof}
Suppose $X^* \nsupseteq \ell_1$. By Proposition~\ref{prop-old:dualcontains_l_1} we have that $n(X)<1$, so Corollary~\ref{cor:n'-equal-1} tells us that $X$ is a Hilbert space.
\end{proof}

For the proof of Theorem~\ref{thm:n'-one-implies-Hilbert}, we need a number of auxiliary results. The
first one, which gathers a series of results appearing in \cite{Ros},
allows us to write a Banach space with one-unconditional basis and
non-trivial Lie algebra as the absolute sum of suitable Hilbert
spaces in such a way that the skew-hermitian operators are diagonal. In order to present the result (we include the deduction of its proof as a consequence of the results in \cite{Ros} for
commodity of the reader) we need some notation that we also borrow from \cite{Ros}.

A subspace $H$ of a Banach space $X$ is said to be
\emph{orthogonally-complemented} in $X$ if there is a subspace $Y$
of $X$ such that $X=H\oplus Y$ and for every $h\in H$, $y\in Y$,  $\|h + y\|=\|h-y\|$. In such a case, the space $Y$ is unique, it is called the \emph{orthogonal complement} of $H$, and denoted by
$Y=O(H)$. The subspace $H$ is said to be \emph{well-embedded} if
there is a subspace $Y$ of $X$ such that $X=H\oplus Y$ and for all $h\in
H$, $y\in Y$, and every onto isometry $U\in \mathcal{L}(H)$,
$\|h+y\|=\|Uh+y\|$. If, moreover, $H$ is euclidean, $H$ is called a
\emph{well-embedded Hilbert subspace}. $H$ is said to be a
\emph{Hilbert component} of $X$ if it is a maximal non-zero
well-embedded Hilbert subspace. Hilbert components exist as soon as
there is a well-embedded Hilbert subspace (see
\cite[Theorem~1.12]{Ros}). The \emph{Functional Hilbertian part} of
$X$, denoted by $FH(X)$, is the closed linear span of the union of
all Hilbert components of $X$ with $\dim\geq 2$ and the
\emph{orthogonal part} is $Or(X) = \{x \ : \ [x] \text{ is a Hilbert component of } X\} \cup\{0\}$ where $[x]$ denotes the (closed)
linear span of $x$. $X$ is said to be \emph{pure} if it has no
rank-two operators in $\mathcal{Z}(X)$.

\begin{lemma}\label{lem:Rosenthal-decomposition}
Let $\Gamma$ be a non-empty set and let $X$ be  a Banach space with a one-unconditional basis indexed on the set $\Gamma$. Suppose that $\mathcal{Z}(X)\neq \{0\}$. Then, there are a non-empty set
$\Lambda$, a subset $\Gamma_{Or}\subset \Gamma$, Hilbert spaces $\{H_\alpha \
: \ \alpha\in \Lambda\cup \Gamma_{Or}\}$, and a pure space $V$ with an one-unconditional basis indexed on the set $\Lambda\cup \Gamma_{Or}$ so that
$$
X=\Big[\bigoplus_{\alpha \in \Lambda\cup \Gamma_{Or}}
H_\alpha\Big]_V \,.
$$
Moreover, the following hold:
\begin{enumerate}
\item $\dim(H_\alpha)\geq 2$ for every $\alpha\in \Lambda$.
\item Every $S\in \mathcal{Z}(X)$ is diagonal, i.e.\ for
    every $\alpha \in \Lambda \cup \Gamma_{Or}$ there is $S_\alpha\in \mathcal{Z}(H_\alpha)$  such that
$$
[S(x)](\alpha)=S_\alpha (x(\alpha)) \qquad \big(x \in X\big).
$$
Besides, one has $S_\alpha=0$ for every $\alpha\in \Gamma_{Or}$.
\end{enumerate}
\end{lemma}

\begin{proof}
Since $\mathcal{Z}(X)\neq \{0\}$ and $X$ has a one-unconditional basis, we can use
Theorem~3.10 in \cite{Ros} to obtain that $X$ is not pure. So there is a rank-two
operator in $\mathcal{Z}(X)$. The image of such an operator is a well-embedded
Hilbert subspace of $X$ (see the first lines of page 435 of
\cite{Ros}) and so the Functional Hilbertian part of $X$ is
non-trivial. Since $X$ has a normalized one-unconditional basis, $FH(X)$ is
orthogonally complemented in $X$ (see Remarks~(1) right after
Theorem~3.1 in \cite{Ros}). We use now  \cite[Theorem~3.6]{Ros} to
get the existence of a non-empty set $\Lambda$, a subset $\Gamma_{Or}\subset
\Gamma$, Hilbert spaces  $\{H_\alpha \ : \ \alpha\in \Lambda\cup
\Gamma_{Or}\}$, and a pure space $V$ with a one-unconditional basis indexed on the set $\Lambda\cup
\Gamma_{Or}$ so that the following hold:
\begin{itemize}
\item For every $\alpha\in \Lambda$, $H_\alpha$ is a Hilbert component of $X$ with dimension greater than or equal to two.
\item $\displaystyle
X=\Big[\bigoplus_{\alpha \in \Lambda\cup \Gamma_{Or}}
H_\alpha\Big]_V\,.
$
\end{itemize}
Moreover, Corollary~3.11 (see also Theorem~3.8) in \cite{Ros} tells
us that every $S\in \mathcal{Z}(X)$ is diagonal and $S_\alpha=P_\alpha S I_\alpha=0$ for every $\alpha\in \Gamma_{Or}$.
\end{proof}

\begin{lemma}\label{lem:space-absolute-norm-coord-sum-2}
Let $E\subset \R^\Lambda$ be a Banach space with absolute norm so
that $E'$ is norming. Suppose that there is $\mu\in \Lambda$ such
that $\|\cdot\|_{E_{\{\mu,\lambda\}}}=\|\cdot\|_2$ for every
$\lambda\in \Lambda\setminus\{\mu\}$. If $\nu\in \Lambda\setminus\{\mu\}$, then the
positive operator $U\in \mathcal{L}(E)$ given by $U=e_\mu^*\otimes e_\nu$
satisfies
$$
\|U\|=1 \qquad \text{and}\qquad v(U)<1.
$$
\end{lemma}

\begin{proof}
It is obvious that $\|U\|=1$.
Assume that $v(U)=1$ and use
Remark~\ref{rem:dense-subset-numericalradius} to ensure the
existence of two sequences of positive elements $\{a_n\}_{n\in \N}$ in $S_E$
and $\{b_n\}_{n\in \N}$ in $S_{E'}$ such that $\lim_{n\to \infty}\langle b_n, a_n\rangle=1$
and $\lim_{n\to \infty} \langle b_n, Ua_n\rangle =1$.
Since the family $\sum_{\lambda\in \Lambda} a_n(\lambda) b_n(\lambda)=\langle b_n, a_n\rangle$ is summable
for every $n$,
there is a countable subset $\Lambda_1$ of $\Lambda$ such that
$a_n(\lambda) b_n(\lambda)=0$ for every $\lambda\in \Lambda\setminus
\Lambda_1$ and every $n\in \N$. We will assume that $\mu=1$, $\nu=2$
and $\Lambda_{1}\cup\{\mu, \nu\} = \mathbb{N}$. Since $\{a_n\}$ and
$\{b_n\}$ are coordinate-wise bounded, we may assume that $\lim_{n\to \infty}
a_n(k) = a(k)$ and $\lim_{n\to \infty} b_n(k) = b(k)$ for each $k\in \mathbb{N}$
by considering proper subsequences.
Since $\lim_{n\to \infty} \langle b_n, Ua_n\rangle =1$, we have
$\lim_{n\to \infty} a_n(1)b_n(2)=1$. That is, $a(1)=1$ and $b(2)=1$. Notice that
for $k$ in $\mathbb{N}\setminus\{1\}$, we have
\[
a(1)^2 + a(k)^2 = \lim_{n\to \infty} (a_n(1)^2 + a_n(k)^2) \leq \lim_{n\to \infty}
\|a_n\|^2=1.
\]
Similarly, we have
\[
b(1)^2 + b(2)^2 = \lim_{n\to \infty} (b_n(1)^2 + b_n(2)^2) = \lim_{n\to \infty} \|b_n|_{E_{\{\mu,\nu\}}}\|^2 \leq \lim_{n\to \infty}
\|b_n\|^2= 1.
\]
This shows that $a(k)=0$ for all $k\neq 1$ and $b(1)=0$. Hence
$\lim_{n\to \infty} a_n(k)b_n(k) =0$ for all $k\in \mathbb{N}$.
So we can extract further subsequences and assume that the
sequence $\{a_n(k)b_n(k)\}_{n\in \N}$ is monotone decreasing and convergent
to $0$ for each $k\geq 1$.

Finally, notice that
\begin{align*}
 1&=\lim_{n\to \infty}  \langle b_n, a_n \rangle=\lim_{n\to \infty}\sum_{k=1}^\infty
 a_n(k)b_n(k)=0
\end{align*}
where the last equality holds by the Lebesgue dominated convergence
theorem. This is a contradiction and the proof is complete.
\end{proof}

\begin{proof}[Proof of Theorem~\ref{thm:n'-one-implies-Hilbert}]
By Lemma~\ref{lem:Rosenthal-decomposition} there are a non-empty set
$\Lambda$, a subset $\Gamma_{Or}\subset \Gamma$, Hilbert spaces $\{H_\alpha \,
: \, \alpha\in \Lambda\cup \Gamma_{Or}\}$ such that $\dim(H_\alpha)\geq 2$ for every $\alpha\in \Lambda$, and a pure space $V$ with a one-unconditional basis indexed on the set $\Lambda\cup
\Gamma_{Or}$ so that
$$
X=\Big[\bigoplus_{\alpha \in \Lambda\cup \Gamma_{Or}}
H_\alpha\Big]_V \,.
$$
Moreover, every $S\in \mathcal{Z}(X)$ is diagonal and $X^*=\Big[\bigoplus_{\alpha \in \Lambda\cup \Gamma_{Or}}
H^*_\alpha\Big]_{V^*}$.

Our goal is to show that $\Lambda\cup \Gamma_{Or}$ is a singleton.
Since $\Lambda$ is non-void we may and do fix $\mu \in \Lambda$.
Suppose for contradiction that there is $\nu\in \Lambda\cup
\Gamma_{Or}$ with $\nu\neq \mu$. We write $V_{\{\mu,\nu\}}$ for the linear
span of $\{e_\mu, e_\nu\}$ and we consider the
constants of equivalence
$$
M=\max_{ae_\mu+be_\nu \in B_{V_{\{\mu,\nu\}}}}\sqrt{a^2+b^2} \qquad
\text{and} \qquad K=\max_{c^2+d^2\leq 1}
\|ce_\mu+de_\nu\|_{V_{\{\mu,\nu\}}}
$$
which obviously satisfy $M\geq1$ and $K\geq 1$.

\emph{Claim: $M\leq 1$ and $K\leq 1$. Therefore, we have that
$\|\cdot\|_{V_{\{\mu,\nu\}}}=\|\cdot\|_2$.}

Once the Claim is established the result follows easily.
Indeed, observe that the coordinate $\nu$ is arbitrary so the claim
also applies to every $\alpha\in \Lambda\cup\Gamma_{Or}$ such that
$\alpha\neq \mu$. Therefore, we can apply
Lemma~\ref{lem:space-absolute-norm-coord-sum-2} for $V$ to get that
the positive operator $U=e^*_\mu\otimes e_\nu$ satisfies
$\frac{v(U)}{\|U\|}<1$. Now
Proposition~\ref{prop:derived-index-positive-operators} tells us that
$n'(X)\leq \frac{v(U)}{\|U\|}<1$ which is a contradiction. So
$\Lambda\cup \Gamma_{Or}=\{\mu\}$ and $X=H_\mu$ is a Hilbert space.

\emph{Proof of the Claim:}
Since $\dim(H_\mu)\geq 2$ and $H_\nu$ is non-trivial we can fix orthogonal elements $u_1,u_2\in S_{H_\mu}$ and $u_3\in S_{H_\nu}$.

\emph{Step 1: $M\leq 1$.}\ Consider the operator
$T_1\in \mathcal{L}(X)$ given by
$$
T_1(x)=\frac{1}{M} I_\mu\big(( u_1 \mid x(\mu)) u_1+
( u_3 \mid x(\nu) ) u_2\big)\quad \big(x\in
X\big),
$$
and observe that, for $x\in S_X$, we have
$$
\|T_1(x)\|=\frac{1}{M}\sqrt{( u_1\mid x(\mu))^2+
( u_3\mid x(\nu) )^2}\leq\frac{1}{M}\sqrt{\|x(\mu)\|^2+
\|x(\nu)\|^2}\leq 1
$$
and, therefore, $\|T_1+\mathcal{Z}(X)\|\leq\|T_1\|\leq1$.
Let us show now that $\|T_1+\mathcal{Z}(X)\|\geq 1$. Indeed, fixed $S\in \mathcal{Z}(X)$,
there is $S_\alpha\in \mathcal{Z}(H_\alpha)$ for every $\alpha \in \Lambda
\cup \Gamma_{Or}$ so that $[S(x)](\alpha)=S_\alpha(x(\alpha))$ for
every $x \in X$. Besides, we fix $\theta\in \{-1,1\}$
satisfying $\theta( u_2\mid S_\mu(u_1)) \geq 0$ and we
observe that $( u_1\mid S_\mu(u_1))=0$ since $v(S_\mu)=0$.
Moreover, we take $a_0,b_0\in[0,1]$ satisfying $a_0e_\mu+b_0e_\nu\in
B_{V_{\{\mu,\nu\}}}$  and $M=\sqrt{a_0^2+b_0^2}$ which give
$$
\|I_\mu(a_0u_1)+I_\nu(\theta
b_0
u_3)\|=\|a_0e_\mu+b_0e_\nu\|_{V_{\{\mu,\nu\}}}\leq 1.
$$
Thus, we can write
\begin{align*}
\|T_1+S\|&\geq\left\|[T_1+S]\big(I_\mu(a_0u_1)+I_\nu(\theta
b_0
u_3)\big)\right\|\\&=\left\|\frac{1}{M}I_\mu(a_0u_1+\theta
b_0u_2)+a_0I_\mu(S_\mu(u_1))+\theta b_0I_\nu
(S_\nu(u_3))\right\|\geq \left\|\frac{1}{M}(a_0u_1+\theta
b_0u_2)+a_0S_\mu(u_1)\right\|\\
&=\sqrt{\Big( \frac{1}{M}(a_0u_1+\theta b_0u_2)+a_0S_\mu(u_1)\mid
\frac{1}{M}(a_0u_1+\theta b_0u_2)+a_0S_\mu(u_1)\Big)}\\
&=\sqrt{\frac{a_0^2+b_0^2}{M^2}+ a_0^2\|S_\mu(u_1)\|^2 +
\frac{2\theta a_0b_0}{M} ( u_2\mid S_\mu(u_1)) }\geq
\sqrt{\frac{a_0^2+b_0^2}{M^2}}= 1.
\end{align*}
Taking infimum on $S\in \mathcal{Z}(X)$ we get $\|T_1+\mathcal{Z}(X)\|\geq1$ and so
$\|T_1+\mathcal{Z}(X)\|=1$. Thus, using that $n'(X)=1$, we obtain $v(T_1)=1$.
Therefore, for each $n\in\N$ there are $x_n\in
S_X$ and $x^*_n\in S_{X^*}$ satisfying
\begin{equation}\label{eq:main-thm-radius-1-M}
\langle x^*_n, x_n \rangle =1 \qquad \text{and} \qquad \lim_{n\rightarrow\infty}\big|\langle x^*_n, T_1(x_n)\rangle\big|=v(T_1)=1.
\end{equation}
This gives, in particular, that
$\lim\limits_{n\rightarrow\infty}\big\|T_1(x_n)\big\|=1$ and so
$\lim\limits_{n\rightarrow\infty}\big\|( u_1\mid x_n(\mu)) u_1+
( u_3\mid x_n(\nu)) u_2\big\|=M$. By passing to a
convenient subsequence we may and do assume that there are
$a,b\in[-1,1]$ so that
\begin{equation} \label{eq:Thm-n'=1-exposing-u1-u3-operator-T1}
\lim_{n\rightarrow\infty}( u_1\mid x_n(\mu))=a \qquad
\text{and}\qquad \lim_{n\rightarrow\infty}(u_3\mid
x_n(\nu))=b.
\end{equation}
Now, we can write
\begin{align*}
M&\geq \sqrt{\|x_n(\mu)\|^2+\|x_n(\nu)\|^2}\\
&\geq \sqrt{|(u_1\mid
x_n(\mu))|^2+|(u_3\mid x_n(\nu))|^2}=\big\|(u_1\mid x_n(\mu))u_1+ ( u_3\mid
x_n(\nu)) u_2\big\|\longrightarrow M,
\end{align*}
which implies that
$$
\lim_{n\rightarrow\infty}\|x_n(\mu)\|=|a|\,, \qquad
\lim_{n\rightarrow\infty}\|x_n(\nu)\|=|b|,\qquad \text{and}\qquad
\sqrt{a^2+b^2}=M.
$$
Observe that if $a=0$ or $b=0$ we obtain $M\leq1$ so we may and do
assume that $ab\neq0$. Using this and the fact that $u_1$ and $u_3$
strongly expose themselves, we deduce from \eqref{eq:Thm-n'=1-exposing-u1-u3-operator-T1}
that
\begin{equation}\label{eq:thm-n'=1-convergence-x_n(mu)-x_n(nu)}
\lim_{n\rightarrow\infty}x_n(\mu)=a u_1 \qquad \text{and} \qquad
\lim_{n\rightarrow\infty}x_n(\nu)=b u_3.
\end{equation}
For each $n\in\N$ consider the element $\xi_n\in X$ given
by
$$
\xi_n(\alpha)=\begin{cases}
au_1 & \text{ if } \alpha= \mu,\\
bu_3 & \text{ if } \alpha= \nu,\\
x_n(\alpha)& \text{ if } \alpha\notin\{ \mu,\nu\}.
\end{cases}
$$
It is clear that $\lim_{n\rightarrow\infty}\|\xi_n-x_n\|=0$ by \eqref{eq:thm-n'=1-convergence-x_n(mu)-x_n(nu)},
so using \eqref{eq:main-thm-radius-1-M} we obtain
$$
\lim_{n\rightarrow\infty}\big|\langle x^*_n, T_1(\xi_n)\rangle\big|=1.
$$
This, together with $T_1(\xi_n)=\frac{1}{M} I_\mu(au_1+bu_2)$, allows us to write
$$
\lim_{n\rightarrow
\infty}\big|\langle x^*_n(\mu), \frac{1}{M}(a u_1+ bu_2) \rangle \big|=1.
$$
Therefore, using that $\frac{1}{M}(a u_1+ bu_2)\in S_{H_\mu}$ strongly exposes
itself and passing to a subsequence, we can find  $\theta_1 \in
\{-1,1\}$ such that
\begin{equation}\label{eq:thm-n'=1-convergence-x^*_n(mu)}
\lim_{n\rightarrow \infty}x^*_n(\mu)=\frac{\theta_1}{M}(a u_1+
bu_2).
\end{equation}
Finally, for each $n\in \N$ we consider the element
$y_n\in X$ given by
$$
y_n(\alpha)=\begin{cases} \frac{\theta_1}{M}(au_1+b u_2)\|x_n(\mu)\|
& \text{ if } \alpha=\mu,\\
x_n(\alpha) & \text{ if } \alpha\neq \mu,
\end{cases}
$$
which obviously satisfies $\|y_n\|=\|x_n\|=1$.
Observe that
\begin{align*}
1- \langle x^*_n, x_n\rangle &\geq \langle x^*_n, y_n\rangle -\langle x^*_n, x_n\rangle =\langle x^*_n, y_n-x_n\rangle \\
& = \langle x^*_n(\mu), y_n(\mu)-x_n(\mu)\rangle =\langle x^*_n(\mu), y_n(\mu)\rangle-\langle x^*_n(\mu), x_n(\mu)\rangle.
\end{align*}
Taking limit in this inequality and using \eqref{eq:main-thm-radius-1-M}, \eqref{eq:thm-n'=1-convergence-x_n(mu)-x_n(nu)} and \eqref{eq:thm-n'=1-convergence-x^*_n(mu)}, we obtain
\begin{align*}
0&\geq \lim_{n\to \infty} \langle x^*_n(\mu), y_n(\mu)\rangle-\lim_{n\to \infty}\langle x^*_n(\mu), x_n(\mu)\rangle= \frac{\theta_1^2}{M^2}\|au_1+bu_2\|^2|a|-\frac{\theta_1a^2}{M}=|a|-\frac{\theta_1a^2}{M}\,.
\end{align*}
Therefore, it follows that
$$
|a|\leq \frac{\theta_1a^2}{M}\leq
\frac{a^2}{M}\leq\frac{|a|}{M}
$$
which forces $M\leq 1$ since $a\neq0$.

\emph{Step 2: $K\leq 1$.}\ Consider the operator $T_2\in
\mathcal{L}(X)$ given by
$$
T_2(x)=\frac{1}{K}\big(( u_1\mid x(\mu)) I_\mu(u_1)+
( u_2\mid x(\mu)) I_\nu(u_3)\big)\quad \big(x\in
X\big).
$$
For $x\in S_X$ we have that $\sqrt{|( u_1\mid
x(\mu))|^2+ |( u_2\mid x(\mu) )|^2}\leq\|x(\mu)\|\leq
1$, so using the definition of $K$ we obtain

$$
\|T_2(x)\|=\frac{1}{K}\big\||( u_1\mid x(\mu))|e_\mu+
|( u_2\mid x(\mu))| e_\nu\big\|_{V_{\{\mu,\nu\}}}\leq \frac{1}{K}\max_{c^2+d^2\leq 1}
\|ce_\mu+de_\nu\|_{V_{\{\mu,\nu\}}}= 1
$$
which gives $\|T_2+\mathcal{Z}(X)\|\leq \|T_2\|\leq 1$. We show now that
$\|T_2+\mathcal{Z}(X)\|\geq 1$: fixed $S\in \mathcal{Z}(X)$, there is $S_\alpha\in
\mathcal{Z}(H_\alpha)$ for every $\alpha \in \Lambda \cup \Gamma_{Or}$ so that
$[S(x)](\alpha)=(S_\alpha (x(\alpha))$ for every $x \in X$. We
take $\theta_2\in \{-1,1\}$ such that $\theta_2( u_1\mid
S_\mu(u_2)) \geq 0$ and $c_0,d_0\in[0,1]$ satisfying
$\sqrt{c_0^2+d_0^2}=1$ and $\|(c_0,d_0)\|_{V_{\{\mu,\nu\}}}=K$. Then
$\|c_0u_1+\theta_2 d_0u_2\|=1$ and so we can write
\begin{align*}
\|T_2+S\|&\geq\|[T_2+S](I_\mu(c_0u_1+\theta_2
d_0u_2))\|=\left\|\frac{c_0}{K}I_\mu (u_1)+\frac{\theta_2
d_0}{K}I_\nu(u_3)+I_\mu(S_\mu(c_0u_1+\theta_2 d_0u_2))\right\|.
\end{align*}
Observe that $( u_1\mid S_\mu(u_1) )=0$ since $S_\mu\in \mathcal{Z}(H_\mu)$. So
\begin{align*}
\Big\|\frac{c_0}{K}u_1+c_0S_\mu(u_1)+&\theta_2
d_0S_\mu(u_2)\Big\|^2\\&=\left\|\frac{c_0}{K}u_1\right\|^2+
\|c_0S_\mu(u_1)+\theta_2
d_0S_\mu(u_2)\|^2+2\Big(\frac{c_0}{K}u_1\mid c_0S_\mu(u_1)+\theta_2
d_0S_\mu(u_2)\Big)\\
&=\left\|\frac{c_0}{K}u_1\right\|^2+
\|c_0S_\mu(u_1)+\theta_2 d_0S_\mu(u_2)\|^2+2\frac{\theta_2
c_0d_0}{K}( u_1\mid S_\mu(u_2))\geq
\left\|\frac{c_0}{K}u_1\right\|^2
\end{align*}
and, therefore,
$$
\left\|\frac{c_0}{K}I_\mu (u_1)+I_\mu(S_\mu(c_0u_1+\theta_2 d_0u_2))\right\|\geq \left\|\frac{c_0}{K}I_\mu (u_1)\right\|.
$$
Using this we can continue the estimation above as follows:
\begin{align*}
\|T_2+S\|&\geq\left\|\frac{c_0}{K}I_\mu (u_1)+\frac{\theta_2
d_0}{K}I_\nu(u_3)+I_\mu(S_\mu(c_0u_1+\theta_2 d_0u_2))\right\|\\
&\geq\left\|\frac{c_0}{K}I_\mu (u_1)+\frac{\theta_2
d_0}{K}I_\nu(u_3)\right\|=\frac{1}{K}\|(c_0,d_0)\|_{V_{\{\mu,\nu\}}}=1.
\end{align*}
Taking infimum on $S\in \mathcal{Z}(X)$, we obtain $\|T_2+\mathcal{Z}(X)\|\geq 1$ and so
$\|T_2+\mathcal{Z}(X)\|=1$. Therefore, the fact that $n'(X)=1$ tells us that $v(T_2)=1$. Hence,
we can find
$x_n\in S_X$ and $x^*_n \in S_{X^*}$ satisfying
\begin{equation}\label{eq:main-thm-radius-1-K}
\langle x^*_n, x_n \rangle=1 \qquad \text{and} \qquad  \lim_{n\rightarrow\infty}\big|\langle x^*_n, T_2(x_n)\rangle\big|=v(T_2)=1.
\end{equation}
This gives in particular that
$\lim_{n\rightarrow\infty}\big\|T_2(x_n)\big\|=1$ and so
$$
\lim_{n\rightarrow\infty}\big\| |( u_1\mid x_n(\mu))| e_\mu+
|( u_2\mid x_n(\mu) )|e_\nu \big\|_{V_{\{\mu,\nu\}}}=K.
$$
By passing to a convenient subsequence we may and do assume that there are
$c,d\in[-1,1]$ so that
$$
\lim_{n\rightarrow\infty}( u_1\mid x_n(\mu))=c \qquad
\text{and}\qquad \lim_{n\rightarrow\infty}( u_2\mid
x_n(\mu))=d.
$$
Therefore, $\|ce_\mu+de_\nu\|_{V_{\{\mu,\nu\}}}=K$. Besides,
as $|( u_1\mid x_n(\mu))|^2+|( u_2\mid
x_n(\mu))|^2\leq \|x_n(\mu)\|^2\leq 1$ we also have that
$\sqrt{c^2+d^2}\leq 1$. In fact one obtains $\sqrt{c^2+d^2}=1$ since
$\|ce_\mu+de_\nu\|_{V_{\{\mu,\nu\}}}=K$.
Observe now that $\|x_n(\mu)\|\leq 1$ and
$$
\lim_{n\rightarrow\infty}( cu_1+du_2\mid
x_n(\mu))=c^2+d^2=1
$$
so, taking into account that $cu_1+du_2\in S_{H_\mu}$ strongly
exposes itself, we get that
$$
\lim_{n\rightarrow\infty}x_n(\mu) = cu_1+du_2.
$$
For each $n\in\N$ consider the element $w_n\in X$ given
by
$$
w_n(\alpha)=\begin{cases}
cu_1+du_2 & \text{ if } \alpha= \mu,\\
x_n(\alpha)& \text{ if } \alpha\neq \mu.
\end{cases}
$$
It is clear that $\lim_{n\rightarrow\infty}\|w_n-x_n\|=0$. So,
using \eqref{eq:main-thm-radius-1-K} we obtain
$$
\lim_{n\rightarrow\infty}\big|\langle x^*_n, T_2(w_n)\rangle\big|=1.
$$
This, together with
$T_2(w_n)=\frac{1}{K}(cI_\mu(u_1)+dI_\nu(u_3))$, allows us
to write
\begin{align}\label{eq:main-thm-desigualdad-A-B-K}
K\leftarrow|c\langle x^*_n(\mu), u_1\rangle +d \langle x^*_n(\nu), u_3 \rangle|&\leq
|c||\langle x^*_n(\mu), u_1\rangle|+|d||\langle x^*_n(\nu), u_3\rangle|\\
&\leq |c|\|x^*_n(\mu)\|+|d|\|x^*_n(\nu)\|\leq K.\notag
\end{align}
By passing to a suitable subsequence we can assume that there are
$a',b'\in[-1,1]$ such that
$$
\lim_{n\rightarrow\infty} \langle x^*_n(\mu), u_1\rangle=a' \qquad \text{and} \qquad
\lim_{n\rightarrow\infty} \langle x^*_n(\nu), u_3\rangle =b'.
$$
Using \eqref{eq:main-thm-desigualdad-A-B-K} we deduce that
$\lim_{n\rightarrow\infty} \|x^*_n(\mu)\|=|a'|$,
$\lim_{n\rightarrow\infty} \|x^*_n(\nu)\|=|b'|$, and
$|c||a'|+|d||b'|=K$. We observe that if $|a'|=0$ then $K\leq 1$ and we
are done, so we assume that $a'\neq0$ from now on.
Since $u_1\in S_{H_\mu}$ strongly exposes itself, $a'\neq0$,
$\lim_{n\rightarrow\infty} \langle x^*_n(\mu),u_1\rangle=a'$, and
$\lim_{n\rightarrow\infty} \|x^*_n(\mu)\|=|a'|$ we obtain that
$\lim_{n\rightarrow\infty} x^*_n(\mu)=a'u_1$.

Finally, for each $n\in \N$ we consider the element
$z_n\in X$ given by
$$
z_n(\alpha)=\begin{cases}
\frac{a'}{|a'|}u_1\|x_n(\mu)\| & \text{ if } \alpha=\mu,\\
x_n(\alpha) & \text{ if } \alpha\neq\mu
\end{cases}
$$
which clearly satisfies $\|z_n\|=\|x_n\|$. Observe that
\begin{align*}
1- \langle x^*_n, x_n\rangle &\geq \langle x^*_n, z_n\rangle -\langle x^*_n, x_n\rangle =\langle x^*_n, z_n-x_n\rangle \\
& = \langle x^*_n(\mu), z_n(\mu)-x_n(\mu)\rangle =\langle x^*_n(\mu), z_n(\mu)\rangle-\langle x^*_n(\mu), x_n(\mu)\rangle.
\end{align*}
Taking limit in this inequality, using \eqref{eq:main-thm-radius-1-K}, and
recalling that  $\lim_{n\rightarrow\infty}\|x_n(\mu)\|=1$,
$\lim_{n\rightarrow\infty}x^*_n(\mu)=a'u_1$, and
$\lim_{n\rightarrow\infty}x_n(\mu) = cu_1+du_2$, we can write
\begin{align*}
0\geq \lim_{n\to\infty}\langle x^*_n(\mu), z_n(\mu)\rangle-\lim_{n\to\infty}\langle x^*_n(\mu), x_n(\mu)\rangle= \Big( a'u_1\mid \frac{a'}{|a'|}u_1\Big)- ( a'u_1\mid cu_1+du_2)= |a'|-a'c.
\end{align*}
Therefore, it follows that
$|a'|\leq a'c$. This obviously implies $|c|=1$ which, together with
$\sqrt{c^2+d^2}=1$, gives $d=0$ and hence $K=1$. So the claim
is established.
\end{proof}

\section{Some open problems}\label{sect:open_problems}

We would like to finish the paper presenting some interesting open problems in the subject.

First, we do not know which are the possible values of the second numerical index when restricted to Banach spaces with numerical index zero.

\begin{problem}
Which is the set of values of $n'(X)$ for Banach spaces $X$ with $n(X)=0$?
\end{problem}

In particular, we have very few information for three-dimensional spaces.

Next, we do not know whether the inequalities for the second numerical index that we get for absolute sums can be extended to $\ell_2$-sums (that is, whether the fact that skew-hermitian operators are diagonal is essential in the proofs). In particular, the following is unknown.

\begin{problem}
Let $Y$, $W$ be Banach spaces. Is $n'(Y\oplus_2 W)\leq \min\{n'(Y),n'(W)\}$?
\end{problem}

It is also unknown whether the results about vector valued function spaces (section~\ref{sec:vector-valued}) can be extended to other spaces as $L_p(\mu,X)$ or, even more, to general K\"{o}the Bochner spaces.

\begin{problem}
Let $X$ be a Banach space and let $E$ be K\"{o}the spaces. Does the inequality $n'(E(X))\leq n'(X)$ hold? In particular, if $\mu$ is a positive measure and $1<p<\infty$, is it true that $n'(L_p(\mu,X))\leq n'(X)$?
\end{problem}

The relation between the second numerical index and the duality has not been completely determined. In particular,

\begin{problem}
Is $n'(X^*)\leq n'(X)$ for every Banach space $X$?
\end{problem}

We do not know whether the results of section~\ref{sect-numindex-one} can be extended to general Banach spaces.

\begin{problem}
Are Hilbert spaces the unique Banach spaces with numerical index zero and second numerical index one?
\end{problem}

Let $X$ be a \textbf{complex} Banach space and write $X_\R$ to denote the underlying real space (i.e.\ $X$ viewed as a real space). Then, $n(X_\R)=0$ and, moreover, $\mathcal{Z}(X_\R)$ is not empty as it contains the multiplication by the complex unit.

\begin{problem}
Which information gives $n'(X_\R)$ on a complex Banach space $X$?
\end{problem}

In particular, the following particular case is interesting.

\begin{problem}
Let $\oplus_a$ be an absolute sum and consider $X=\C \oplus_a \C$. What is the value of $n'(X_\R)$?
\end{problem}

For instance, if $\oplus_a=\oplus_2$, then $n'(X_\R)=1$, while if $\oplus_a=\oplus_\infty$ or $\oplus_a=\oplus_1$, then $1/2\leq n'(X_\R)\leq\frac{\sqrt{3}}{2}$ (see Example~\ref{example:HoplusinftyH-Hoplus1H}).

Finite-dimensional real spaces with numerical index zero were characterized in \cite{MaMeRo}, where a structure result is given. By Proposition~\ref{prop:finite-dimensional-n'>0}, the second numerical index of all of them is positive. It would be interesting to get results in this line.

\begin{problem}
Study the second numerical index of finite-dimensional real spaces with numerical index zero.
\end{problem}


\begin{thebibliography}{99}

\bibitem{Ardalani} \textsc{M.~A.~Ardalani}, Numerical index with respect to an operator, \emph{Studia Math.} \textbf{224} (2014), 165--171.

\bibitem{AvKadMarMerShe-SCD}
\textsc{A.~Avil\'es, V.~Kadets, M.~Mart\'{\i}n, J.~Mer\'{\i}, and V.~Shepelska},
Slicely countably determined Banach spaces,
\emph{Trans. Amer. Math. Soc.} \textbf{362} (2010), 4871--4900.

\bibitem{B-D1} \textsc{F.~F.~Bonsall and J.~Duncan},
\emph{Numerical Ranges of operators on normed spaces and of elements
of normed algebras}, London Math. Soc. Lecture Note Series
\textbf{2}, Cambridge University Press, 1971.

\bibitem{B-D2} \textsc{F.~F.~Bonsall and J.~Duncan},
\emph{Numerical Ranges II}, London Math. Soc. Lecture Note Series
\textbf{10}, Cambridge University Press, 1973.

\bibitem{BKMW} \textsc{K.~Boyko, V.~Kadets, M.~Mart\'{\i}n,
D.~Werner}, Numerical index of Banach spaces and duality,
\emph{Math. Proc. Cambridge Phil. Soc.} \textbf{142} (2007),
93--102.

\bibitem{Cabrera-Rodriguez} \textsc{M.~Cabrera and A.~Rodr\'{\i}guez Palacios}, \emph{Non-associative normed algebras}, volume 1: the Vidav-Palmer and Gelfand-Naimark Theorems, Encyclopedia of Mathematics and Its Applications \textbf{154}, Cambridge Univesity press, 2014.

\bibitem{ChicaMartinMeri-QJM2014} \textsc{M.~Chica, M.~Mart\'{\i}n, and J.~Mer\'{\i}}, Numerical radius of rank-1 operators on Banach spaces, \emph{Quart. J. Math.} \textbf{65} (2014), 89--100.

\bibitem{ChoiGarciaMaestreMartin-QJ} \textsc{Y.~S.Choi, D.~Garc\'{i}a, M.~Maestre and M.~Mart\'{i}n}, The polynomial numerical index for some complex vector-valued function spaces, \emph{Quart. J. Math.} {\bf 59} (2008), 455--474.

\bibitem{D-Mc-P-W} \textsc{J.~Duncan, C.~M.~McGregor, J.~D.~Pryce, and
A.~J.~White}, \emph{The numerical index of a normed space}, J.
London Math. Soc. (2) \textbf{2} (1970), 481--488.

\bibitem{F-M-P} \textsc{C.~Finet, M.~Mart\'{\i}n, and R.~Pay\'{a}},
Numerical index and renorming, \emph{Proc. Amer. Math. Soc.}
\textbf{131} (2003), 871--877.

\bibitem{Godefroy} \textsc{G.~Godefroy},
Existence and uniqueness of isometric preduals: a survey. \emph{Banach
space theory (Iowa City, IA, 1987)}, 131--193, Contemp. Math. \textbf{85},
Amer. Math. Soc., Providence, RI, 1989.

\bibitem{Godefroy2014} \textsc{G.~Godefroy},
Uniqueness of Preduals in Spaces of Operators, \emph{Canad. Math. Bull.} \textbf{57} (2014), 810--813.


\bibitem{GuiOle} \textsc{A.~J.~Guirao and O.~Kozhushkina}, The Bishop-Phelps-Bollob\'{a}s property for numerical radius in $\ell_1(\C)$, \emph{Studia Math.} \textbf{218} (2013), 41--54.

\bibitem{HWW} \textsc{P.~Harmand, D.~Werner, and D.~Werner},
    \emph{$M$-ideals in Banach spaces and Banach algebras},
    Lecture Notes in Math. \textbf{1547}, Springer-Verlag,
    Berlin, 1993.

\bibitem{Horn} \textsc{G.~Horn}, Characterization of the predual and ideal structure of a $JBW^*$-triple, \emph{Math. Scand.} \textbf{61} (1987), 117--133.

\bibitem{KaMaPa} \textsc{V.~Kadets. M.~Mart\'{\i}n, and R.~Pay\'{a}},
    Recent
    progress and open questions on the numerical index of
    Banach spaces, \emph{Rev. R. Acad. Cien. Serie A. Mat.}
    \textbf{100} (2006), 155--182.

\bibitem{KLM-BPBp-nu}
\textsc{S.~K.~Kim, H.~J.~Lee, and M.~Mart\'{\i}n}, On the Bishop-Phelps-Bollob\'{a}s Property for Numerical Radius, \emph{Abstr. Appl. Anal.} Volume \textbf{2014}, Article ID 479208, 15 pages.

\bibitem{LopezMartinMeri} \textsc{G.~L\'{o}pez, M.~Mart\'{\i}n, and J.~Mer\'{\i}},
Numerical index of Banach spaces of weakly or weakly-star continuous functions, \emph{Rocky Mount. J. Math.} \textbf{38} (2008), 213--223.

\bibitem{MarJFA2008} \textsc{M.~Mart\'{\i}n}, The group of isometries of a Banach space and duality, \emph{J. Funct.
    Anal.} \textbf{255} (2008), 2966--2976.

\bibitem{Mar-numrange-JMAA2016} \textsc{M.~Mart\'{\i}n}, On different definitions of numerical range, \emph{J. Math. Anal. Appl.} \textbf{433} (2016), 877--886.

\bibitem{MarMerLAMA2009} \textsc{M.~Mart\'{\i}n and J.~Mer\'{\i}}, A note on the numerical index of the $L_p$ space of dimension two, \emph{Linear Mutl. Algebra}
    \textbf{57} (2009), 201--204.

\bibitem{MarMerPopRan} \textsc{M.~Mart\'{\i}n, J.~Mer\'{\i}, M.~Popov, and B.~Randrianantoanina}, Numerical index of absolute sums of Banach spaces, \emph{J. Math. Anal. Appl.}
    \textbf{375} (2011), 207--222.

\bibitem{MaMeRo} \textsc{M.~Mart\'{\i}n, J.~Mer\'{\i}, and A.~Rodr\'{\i}guez-Palacios},
Finite-dimensional Banach spaces with numerical index zero,
\emph{Indiana University Math. J.} \textbf{53} (2004), 1279--1289.

\bibitem{MarPay} \textsc{M.~Mart\'{\i}n and R.~Pay\'{a}}, Numerical index of vector-valued function spaces, \emph{Studia Math.} \textbf{142} (2000), no. 3, 269--280.

\bibitem{MarVil} \textsc{M.~Mart\'{\i}n and A.~Villena},
Numerical index and the Daugavet property for $L_\infty(\mu,X)$,
\emph{Proc. Edinb. Math. Soc.} \textbf{46} (2003), no. 2, 415--420.

\bibitem{MPRY} \textsc{J.~F.~Mena, R.~Pay\'{a}, A.~Rodr\'{\i}guez-Palacios, and D.~Yost}, \emph{Absolutely proximinal subspaces of Banach spaces},
J. Aprox. Theory \textbf{65} (1991), 46--72.


\bibitem{Paya82} \textsc{R.~Pay\'{a}},
Numerical range of operators and structure in Banach spaces,
\emph{Quart. J. Math. Oxford} \textbf{33} (1982), 357--364.


\bibitem{Pfitzner} \textsc{H.~Pfitzner}, Separable $L$-embedded Banach spaces are unique preduals, \emph{Bull. London Math. Soc.} \textbf{39} (2007), 1039--1044.


\bibitem{Ros} \textsc{H.~Rosenthal}, Functional hilbertian
sums, \emph{Pac. J. Math.} \textbf{124} (1986), 417--467.

\end{thebibliography}
\end{document}